\documentclass[smallextended,numbook,runningheads]{svjour3}     
\smartqed  
\usepackage{graphicx}
\usepackage{amsmath}
\usepackage{epstopdf} 
\usepackage{mathptmx}      
\usepackage{mathtools}
\usepackage{amsfonts}

\newcommand{\ds}{\mathrm ds}
\newcommand{\dx}{\mathrm dx}
\newcommand{\dt}{\mathrm dt}
\newcommand{\dtheta}{\mathrm d\theta}
\begin{document}

\title{Global-in-time $H^1$-stability of L2-1$_\sigma$ method on general nonuniform meshes for subdiffusion equation\thanks{C. Quan is supported by NSFC Grant 11901281, the Stable Support Plan Program of Shenzhen Natural Science Fund (Program Contract No. 20200925160747003), and Shenzhen Science and Technology Program (Grant No. RCYX20210609104358076).
}}

\titlerunning{Global-in-time $H^1$-stability of L2-1$_\sigma$ method}        

\author{Chaoyu Quan         \and
        Xu Wu 
}


\institute{Chaoyu Quan \at
              SUSTech International Center for Mathematics, Southern University of Science and Technology, Shenzhen, China. \\
              \email{quancy@sustech.edu.cn}         
           \and
           Xu Wu \at
           Department of Mathematics,  Harbin Institute of Technology, Harbin 150001, China; Department of Mathematics, Southern University of Science and Technology, Shenzhen, China. \\
           \email{11849596@mail.sustech.edu.cn}
}

\date{Received: date / Accepted: date}

\maketitle

\begin{abstract}
In this work  the L2-1$_\sigma$ method on general nonuniform meshes is studied for the subdiffusion equation.
When the time step ratio is no less than $0.475329$, a  bilinear form associated with the L2-1$_\sigma$ fractional-derivative operator is proved to be positive semidefinite and a new global-in-time $H^1$-stability of L2-1$_\sigma$ schemes is then derived under simple assumptions on the initial condition and the source term.
In addition, the sharp $L^2$-norm convergence is proved under the constraint that the time step ratio is no less than $0.475329$. 
\keywords{Subdiffusion equation \and L2-1$_\sigma$ method \and Nonuniform meshes  \and $H^1$-stability \and Convergence}
\subclass{35R11 \and 65M12 }
\end{abstract}

\sloppy
\section{Introduction}
In the past decade, many numerical methods have been proposed to solve the time-fractional diffusion equations~\cite{metzler2000random,gorenflo2002time}. 
If the solution is sufficiently smooth (which requires the initial value to be smooth
and satisfying some compatibility conditions), the  L1 scheme has $(2-\alpha)$ order accuracy, see the works of Langlands and Henry~\cite{langlands2005accuracy}, Sun-Wu~\cite{sun2006fully}, and Lin-Xu~\cite{lin2007finite}.
Alikhanov proposes the L2-1$_\sigma$ scheme having second order accuracy in time ~\cite{alikhanov2015new}. Gao-Sun-Zhang study an L2 method of $(3-\alpha)$-order on uniform meshes in~\cite{gao2014new} and later a slightly different L2 method is analyzed by Lv-Xu in~\cite{lv2016error}.  In addition to the Lagrangian interpolation methods, the discontinous Galerkin methods is analyzed by  Mustapha-Abdallah-Furati ~\cite{mustapha2014discontinuous} and the convolution quadrature (CQ) scheme is studied by  Jin-Li-Zhou~\cite{jin2017correction}, both of which can obtain the desired high-order accuracy.

However, simple examples show that for given smooth data, the solutions to time-fractional problems typically have weak singularities. 
Some works start to focus on the numerical solution of more typical fractional problems whose solutions exhibit weak singularities.
In particular, the L1, L2-1$_\sigma$, and L2 methods on the graded meshes have been developed. 
Stynes-Riordan-Gracia~\cite{stynes2017error} prove the sharp error analysis of L1 scheme on graded meshes. 
Kopteva provides a different analysis framework of the L1 scheme on graded meshes in two and three spatial dimensions in ~\cite{kopteva2019error}.
Chen-Stynes~\cite{chen2019error} prove the second-order convergence of the L2-1$_\sigma$ scheme on fitted meshes combining the graded meshes and quasiuniform meshes.
Kopteva-Meng~\cite{kopteva2020error}  provide sharp pointwise-in-time error bounds for quasi-graded termporal meshes with arbitrary degree of grading for the L1 and L2-1$_\sigma$ schemes. 
Later Kopteva generalize this sharp pointwise error analysis to an L2-type scheme on quasi-graded meshes~\cite{kopteva2021error}. 
Liao-Li-Zhang establish the sharp error analysis for the L1 scheme of subdiffusion equation on general nonuniform meshes in~\cite{liao2018sharp} and then Liao-Mclean-Zhang study the L2-1$_\sigma$ scheme in~\cite{liao2019discrete,liao2018second}, where a discrete Gr{\`o}nwall inequality is introduced.
This analysis for general nonuniform meshes can be used to design adaptive strategies of time steps.

Taking into account the singularity of exact solution, Mustapha-Abdallah-Furati~\cite{mustapha2014discontinuous} analyze the global high-order convergence of the discontinuous Galerkin method for subdiffusion equation on graded mesh.
The CQ methods provides a flexible framework for constructing high-order methods to approximate the fractional
derivative, developed by Lubich in~\cite{lubich1986discretized,lubich1988convolution,lubich2004convolution}.  
Along this way, Lubich-Sloan-Thom{\'e}e~\cite{lubich1996nonsmooth} analyze first and second order CQ schemes for  subdiffusion equation. 
In recent years, Jin-Li-Zhou~\cite{jin2017correction,jin2020subdiffusion}  combine  BDF (backward differentiation formula) CQ methods with corrections to achieve higher (more than two) order convergence  which can also overcome the weak singularity problem for time-fractional diffusion equation. 
Banjai and L{\'o}pez-Fern{\'a}ndez~\cite{banjai2019efficient} provide  an arbitrarily high-order accuracy algorithm for subdiffusion equation based on Runge-Kutta CQ. 
In addition, the CQ methods have also been developed  to solve nonlinear subdiffusion equations (see~\cite{jin2018numerical,al2019numerical,wang2020high,li2022exponential}).

In this work, we first study the $H^1$-stability of the L2-1$_\sigma$ method proposed initially in~\cite{alikhanov2015new} on general nonuniform meshes for subdiffusion equation with homogeneous Dirichlet boundary condition:
\begin{equation*}
    \partial_t^\alpha u(t,x) =\Delta u(t,x)+f(t,x),\quad (t,x)\in (0,\infty)\times\Omega,  
\end{equation*}
where $\Omega$ is a bounded Lipschitz domain in $\mathbb R^d$. 
For the L2-1$_\sigma$ fractional-derivative operator denoted by $L_k^{\alpha,*}$, we prove that the following bilinear form
\begin{equation}\label{eq:Bn}
    \mathcal B_n(v,w) = \sum_{k=1}^{n}\langle L_k^{\alpha,*} v, \delta_k w\rangle,\quad \delta_k w \coloneqq w^k-w^{k-1},~n\geq 1,
\end{equation}
is positive semidefinite under the restrictions \eqref{thm1cond} on time step ratios $\rho_k \coloneqq \tau_k/\tau_{k-1}$ with $\tau_k$ the $k$th time step and $k\geq2$.
In fact, the positive semidefiniteness of $\mathcal B_n$ on general nonuniform meshes is an open problem as stated in the conclusion of~\cite{liao2020second}, where the maximum principle and convergence analysis are provided for L2-1$_\sigma$ scheme of the time-fractional Allen--Cahn equation but not the positive definiteness of L2-1$_\sigma$ operator.
On the positive definiteness, Karaa presents in~\cite{karaa2021positivity,al2022time} a general criteria ensuring the positivity of quadratic forms that can be applied to the time-fractional operators such as the L1 formula.
In~\cite{liao2021energy}, Liao-Tang-Zhou proves the positive definiteness of a new L1-type operator. 

Based on the positive semidefiniteness of $\mathcal B_n$ associated with L2-1$_\sigma$ operator, we propose a new \emph{global-in-time} $H^1$-stability result in Theorem~\ref{thm2} for the L2-1$_\sigma$ scheme. In particular, when $\rho_k\ge 0.475329$ for $k\geq 2$, the restrictions \eqref{thm1cond} hold and the $H^1$-stability  can be ensured for all time. 
 
Besides the global-in-time $H^1$-stability of the L2-1$_\sigma$ scheme in Theorem~\ref{thm2}, we revisit the sharp convergence analysis in~\cite{liao2018second}  by Liao-Mclean-Zhang.
We provide a proof of sharp $L^2$-norm convergence based on new properties of the L2-1$_\sigma$ coefficients, where the restriction on time step ratios is relaxed from $\rho_k\geq 4/7$ in~\cite{liao2018second} to $\rho_k\geq 0.475329$.

In the numerical implementations, we compare the L2-1$_\sigma$ schemes on the standard graded meshes~\cite{stynes2017error} and the $r$-variable graded meshes (with varying grading parameter) proposed in~\cite{quan2022h}.
According to our stability analysis, these methods are all $H^1$-stable. 
In our example, it can be observed that choosing proper $r$-variable graded meshes can lead to better numerical performance.

This work is organized as follows. 
In Section~\ref{sect2},  the derivation, explicit expression and  reformulation of L2-1$_\sigma$ fractional-derivative operator are provided. 
In Section~\ref{sect3}, we prove the positive semidefiniteness of the bilinear form $\mathcal B_n$ under some mild restrictions on the time step ratios. 
In Section~\ref{sect4}, we establish a new global-in-time $H^1$-stability of the L2-1$_\sigma$ scheme for the subdiffusion equation, based on the positive semidefiniteness result. Moreover we show the global error estimate when $\rho_k\geq 0.475329$ under low regularity assumptions on the exact solution.
In Section~\ref{sect5}, we do some first numerical tests.

\section{Discrete fractional-derivative operator}\label{sect2}
In this part we show the derivation, explicit expression and reformulation  of L2-1$_\sigma$ operator on an arbitrary nonuniform mesh.

We consider the L2-1$_\sigma$ approximation of the fractional-derivative operator defined by
\begin{equation*}
    \partial_t^\alpha u = \frac{1}{\Gamma(1-\alpha)} \int_0^t \frac{u'(s)}{(t-s)^\alpha} \, \ds. 
\end{equation*}
Take a nonuniform time mesh $0 = t_0<t_1<\ldots<t_{k-1}<t_k<\ldots$ with $k\geq 1$. 
Let $\tau_j = t_j-t_{j-1}$ and $\sigma=1-\alpha/2$ (c.f.~\cite{alikhanov2015new} for this setting of $\sigma$). 
 The fractional derivative $\partial_t^\alpha u(t)$ at $t = t_k^*\coloneqq t_{k-1}+\sigma\tau_k$ could be approximated by the following L2-1$_\sigma$ fractional-derivative operator
\begin{align}\label{eq:Lk0}
    L_k^{\alpha,*} u 
    & = \frac{1}{\Gamma(1-\alpha)} \left(
    \sum_{j=1}^{k-1}\int_{t_{j-1}}^{t_j} \frac{\partial_s H_2^j(s)}{(t_k^*-s)^\alpha}\,\ds + \int_{t_{k-1}}^{t_k^*} \frac{\partial_s H_1^{k}(s)}{(t_k^*-s)^\alpha}\,\ds \right) \\\nonumber
    & = \frac{1}{\Gamma(1-\alpha)}\left( \sum_{j=1}^{k-1} (a_{j}^k u^{j-1} + b_{j}^k u^j + c_{j}^k u^{j+1} ) 
    \right)+\frac{\sigma^{1-\alpha}(u^k-u^{k-1})}{\Gamma(2-\alpha)\tau_k^{\alpha}}, \nonumber
\end{align}   
where for $1\leq j\leq k-1$,
\begin{align*}
H_2^j (t) =& \frac{(t-t_j)(t-t_{j+1})}{(t_{j-1}-t_j)(t_{j-1}-t_{j+1})} u^{j-1} + \frac{(t-t_{j-1})(t-t_{j+1})}{(t_{j}-t_{j-1})(t_{j}-t_{j+1})} u^{j} \\ 
& + \frac{(t-t_{j-1})(t-t_{j})}{(t_{j+1}-t_{j-1})(t_{j+1}-t_{j})} u^{j+1},\\
H_1^k (t) =& \frac{t-t_k}{t_{k-1}-t_k} u^{k-1} + \frac{t-t_{k-1}}{t_{k}-t_{k-1}} u^k,
\end{align*}
and 
\begin{equation}
\begin{aligned}\label{eq:aj}
   a_{j}^k & =  \int_{t_{j-1}}^{t_j} \frac{2s -t_j-t_{j+1}}{\tau_{j}(\tau_{j}+\tau_{j+1})} \frac{1}{(t_k^*-s)^\alpha}\,\ds 
   = \int_0^1 \frac{-2 \tau_j(1-\theta)-\tau_{j+1}}{(\tau_{j}+\tau_{j+1})(t_k^*-(t_{j-1}+\theta \tau_j))^\alpha}\,\dtheta,\\
    b_{j}^k  & = -\int_{t_{j-1}}^{t_j} \frac{2s -t_{j-1}-t_{j+1}}{\tau_{j}\tau_{j+1}} \frac{1}{(t_k^*-s)^\alpha}\,\ds 
   = -\int_0^1 \frac{2 \tau_j\theta-\tau_j-\tau_{j+1}}{\tau_{j+1}(t_k^*-(t_{j-1}+\theta \tau_j))^\alpha}\,\dtheta, \\
    c_{j}^k  & =  \int_{t_{j-1}}^{t_j} \frac{2s -t_{j-1}-t_{j}}{\tau_{j+1}(\tau_{j}+\tau_{j+1})} \frac{1}{(t_k^*-s)^\alpha}\,\ds 
   = \int_0^1 \frac{\tau_j^2(2\theta-1)}{\tau_{j+1}(\tau_{j}+\tau_{j+1})(t_k^*-(t_{j-1}+\theta \tau_j))^\alpha}\,\dtheta.
\end{aligned}
\end{equation}
It can be verified that $a_{j}^k<0$, $b_{j}^k>0$, $c_{j}^k>0$, and $a_{j}^k+b_{j}^k+c_{j}^k =0$ for $1\leq j\leq k-1$.

Specifically speaking, we can figure out the explicit expressions of $a_j^k$ and $c_j^k$ as follows (note that $b_j^k = -a_j^k-c_j^k$): for $1\leq j\leq k-1$,
\begin{align*}
   a^{k}_j
   & = \frac{\tau_{j+1}}{(1-\alpha)\tau_j(\tau_j+\tau_{j+1})}(t_k^*-t_j)^{1-\alpha}-\frac{2\tau_{j}+\tau_{j+1}}{(1-\alpha)\tau_j(\tau_j+\tau_{j+1})}(t_k^*-t_{j-1})^{1-\alpha}\\
   &\quad +\frac{2}{(2-\alpha)(1-\alpha)\tau_j(\tau_j+\tau_{j+1})}\left[(t_k^*-t_{j-1})^{2-\alpha}-(t_k^*-t_j)^{2-\alpha}\right],\\
     c^{k}_j  
   &= \frac{1}{(1-\alpha)\tau_{j+1}(\tau_j+\tau_{j+1})}
   \Big[-\tau_j( (t_k^*-t_{j-1})^{1-\alpha}+ (t_k^*-t_{j})^{1-\alpha})\\
   & \quad+2(2-\alpha)^{-1} ( (t_k^*-t_{j-1})^{2-\alpha}- (t_k^*-t_{j})^{2-\alpha}) \Big].
\end{align*}
We reformulate the discrete fractional derivative $L_k^{\alpha,*}$ in \eqref{eq:Lk0} as
\begin{equation}\label{eq:Lk1}
\begin{aligned}
L_k^{\alpha,*} u
  &=\frac{1}{\Gamma(1-\alpha)}\left( c_{k-1}^k \delta_k u -a_1^k\delta_1 u  +\sum_{j=2}^{k-1} d^k_{j} \delta_j u \right)+\frac{\sigma^{1-\alpha}}{\Gamma(2-\alpha)\tau_k^{\alpha}}\delta_k u, 
    \end{aligned} 
\end{equation}
where $\delta_j u= u^j-u^{j-1},$ $d^k_j\coloneqq c^k_{j-1}-a^k_{j}.$
Here we make a convention that $a_1^1=0$ and $c_0^1=0$. 

To establish the global-in-time $H^1$-stability of L2-1$_\sigma$ method for fractional-order parabolic problem, we shall prove the positive semidefiniteness of $ \mathcal B_n$ defined in \eqref{eq:Bn}.

\section{Positive semidefiniteness of bilinear form $\mathcal B_n$}\label{sect3}
In this section, 
we first propose some properties of the L2-1$_\sigma$ coefficients $a^k_j$, $c^k_j$ and $d^k_j$ in \eqref{eq:Lk1}, which will be useful to establish the positive semidefiniteness of bilinear form $\mathcal B_n$. Then we prove rigorously the positive semidefiniteness of bilinear form $\mathcal B_n$ under some constraints of $\rho_k$, $k\ge 2$.
\begin{lemma}[Properties of $a^k_j$, $c^k_j$ and $d^k_j$]\label{lemmapro}
For the L2-1$_\sigma$ coefficients given in \eqref{eq:Lk1}, given a nonuniform mesh $\{\tau_j\}_{j\geq 1}$,
the following properties hold:
\begin{itemize}
\item[(P1)] $a^k_j<0,~1\leq j\leq k-1,~k\geq 2$;
\item[(P2)] $a^{k+1}_{j}-a^k_{j}>0,~1\leq j\leq k-1,~k\geq 2$; 
\item[(P3)] $a^k_{j+1}-a^k_{j}<0,~1\leq j\leq k-2,~k\geq 3$;
\item[(P4)] $a_{j+1}^k-a_{j}^k<a_{j+1}^{k+1}-a_{j}^{k+1},~1\leq j\leq k-2,~k\geq 3$;
\item[(P5)] $c^k_j>0,~1\leq j\leq k-1,~k\geq 2$;
\item[(P6)] $c^{k+1}_{j}-c^k_{j}<0,~1\leq j\leq k-1,~k\geq 2$; 
\item[(P7)] $d^k_j>0,~2\leq j\leq k-1,~k\geq 3$;
\item[(P8)] $d^{k+1}_{j}-d^k_j<0,~2\leq j\leq k-1,~k\geq 3$.
\end{itemize}
Furthermore, if the nonuniform mesh $\{\tau_j\}_{j\ge 1}$,
 with  $\rho_j \coloneqq \tau_j/\tau_{j-1}$ satisfies
\begin{equation}\label{condition:rho}
    \frac{1}{\rho_{j+1}}\ge\frac{1}{\rho_{j}^2(1+\rho_{j})}-3,\quad \forall j\geq 2,
\end{equation}
then the following properties of $d_j^k$ hold:
\begin{itemize}
\item[(P9)] $d^k_{j+1}-d^k_j>0,~2\leq j\leq k-2,~k\geq 4$;
\item[(P10)] $d_{j+1}^k-d_{j}^k>d_{j+1}^{k+1}-d_{j}^{k+1},~2\leq j\leq k-2,~k\geq 4$.
\end{itemize}

\end{lemma}
\begin{proof}
The proof is the same as the proof of \cite[Lemma 3.1]{quan2022h} except replacing $t_k$ with $t_k^*$. We omit it here.
\end{proof}

\begin{theorem}\label{thm1}
Consider a nonuniform mesh $\{\tau_k\}_{k\geq 1}$ satisfying that $k\ge2$,
\begin{equation}\label{thm1cond}
 \left\{\begin{aligned}
     &\rho_*<\rho_{k+1} \leq  \frac{\rho_k^2(1+\rho_k)}{1-3\rho_k^2(1+\rho_k)},&&  \rho_*<\rho_k< \eta,\\
      &\rho_*<\rho_{k+1},&& \eta\le \rho_k,
\end{aligned}
\right.
\end{equation}
where $ \rho_*\approx 0.356341 $, and $\eta\approx 0.475329$.
Then the  for any function $u$ defined on $[0,\infty)\times\Omega$ and $n\geq 1$,
\begin{equation}\label{eq:PD}
    \mathcal B_n(u,u) = \sum_{k=1}^{n}\langle L_k^{\alpha,*} u, \delta_k u\rangle\geq \sum_{k=1}^{n} \frac{g_k(\alpha)}{2\Gamma(2-\alpha)} \|\delta_k u\|^2_{L^2(\Omega)}\geq 0,
\end{equation}
where 
\begin{equation} \label{eq:gk}
 g_k(\alpha)= \\
\left\{
\begin{aligned}
&  \frac{1}{(\sigma\tau_1)^\alpha}\left(2\sigma-\frac{1-\alpha}{\rho_2^\alpha}\right), &&k=1,\\
      & (1-\alpha)c_{k-1}^k\\
      & \quad +\frac{1}{(\sigma\tau_k)^{\alpha}}\bigg(1-\frac{\alpha(1-\alpha)}{(1+\rho_{k+1})\rho_{k+1}^\alpha}
     \int_0^{1}\frac{s(\rho_{k+1}+s)}{\sigma\rho_{k+1}+s} \,\ds\bigg),  &&2\le k\le n-1,\\
    &(1-\alpha)c_{n-1}^n+ \frac{1}{(\sigma\tau_n)^{\alpha}},&&k=n\neq 2,
\end{aligned}
\right.
\end{equation}
are always positive for $\alpha \in (0,1)$ and $\sigma = 1-\alpha/2$.
\end{theorem}
\begin{proof}
According to \eqref{eq:Lk1}, we can rewrite $\mathcal B_n(u,u)$ in the following matrix form
\begin{equation*}
    \mathcal B_n(u,u) = \sum_{k=1}^{n}\langle L_k^{\alpha,*} u, \delta_k u\rangle=\frac{1}{\Gamma(1-\alpha)}\int_{\Omega}\psi \mathbf M \psi^{\mathrm{T}}\dx,
\end{equation*}
where 
$
    \psi=[\delta_{1} u,\delta_2 u,\cdots,\delta_n u],
$
and
\begin{equation}\label{matM}
\mathbf M=
\begin{pmatrix}
\frac{\sigma^{1-\alpha}}{(1-\alpha)\tau_1^{\alpha}} & \\
-a_1^2& c^2_1+\frac{\sigma^{1-\alpha}}{(1-\alpha)\tau_2^{\alpha}} \\
-a_1^3& d_2^3& c^3_2+\frac{\sigma^{1-\alpha}}{(1-\alpha)\tau_3^{\alpha}} \\
\vdots& \vdots &\ddots & \ddots\\
-a_1^n& d_2^{n}&\cdots&  d_{n-1}^n&c_{n-1}^n+\frac{\sigma^{1-\alpha}}{(1-\alpha)\tau_n^{\alpha}} 
\end{pmatrix}.
\end{equation}
We split $\mathbf M$ as $\mathbf M = \mathbf A+\mathbf B$, where
\begin{equation*}
    \mathbf A=
\begin{pmatrix}
\beta_1 & \\
-a_1^2& \beta_2  \\
-a_1^3 & d^3_2& \beta_3 \\
\vdots& \vdots & \ddots & \ddots\\
-a_1^n& d^{n}_{2}&\cdots&  d^n_{n-1}& \beta_n
\end{pmatrix},
\end{equation*}
and
\begin{equation*}
    \mathbf B={\rm diag}\left(\frac{\sigma^{1-\alpha}}{(1-\alpha)\tau_1^{\alpha}}-\beta_1,~c_1^2+\frac{\sigma^{1-\alpha}}{(1-\alpha)\tau_2^{\alpha}}-\beta_2,~\cdots,~c_{n-1}^n+\frac{\sigma^{1-\alpha}}{(1-\alpha)\tau_n^{\alpha}}-\beta_n\right),
\end{equation*}
with
\begin{equation}\label{betak}
\begin{aligned}
    &2\beta_1= -a^{2}_1,\quad
    2\beta_2 -d^{3}_2=a^{3}_1-a^{2}_1,\\
    &2\beta_k -d^{k+1}_k=d^{k}_{k-1}-d^{k+1}_{k-1},\quad 3\leq k\leq n-1,\\
    &2\beta_n=d^{n}_{n-1},\quad n\geq 3.
\end{aligned}
\end{equation}

Consider the following symmetric matrix 
$
    \mathbf S = \mathbf A+\mathbf A^{\mathrm T}+\varepsilon \mathbf e_n^{\rm T}\mathbf e_n
$
with small constant $\varepsilon>0$ and $\mathbf e_n = (0,\cdots,0,1)\in \mathbb R^{1\times n}$.
According to Lemma~\ref{lemmapro}, if the condition \eqref{condition:rho} holds,  $\mathbf S$ satisfies the following three properties:
\begin{itemize}
\item[{\rm (1)}] $\forall\; 1\leq j < i \leq n$, $\left[ \mathbf S \right]_{i-1,j}\geq \left[ \mathbf S \right]_{i, j}$;
\item[{\rm (2)}] $\forall\; 1 < j \leq i \leq n$, $\left[ \mathbf S \right]_{i, j-1}< \left[ \mathbf S \right]_{i, j}$;
\item[{\rm (3)}] $\forall \;1< j < i \leq n$, $\left[ \mathbf S \right]_{i-1, j-1} - \left[ \mathbf S \right]_{i, j-1}\leq \left[ \mathbf S \right]_{i-1, j} - \left[ \mathbf S \right]_{i, j}$.
\end{itemize}
From \cite[Lemma 2.1]{CSIAM-AM-1-478}, $\mathbf S$ is positive definite. 
Let $\varepsilon \rightarrow 0$. 
We can claim that
$\mathbf A+\mathbf A^{\mathrm T}$
is positive semidefinite. 

In the following we will prove $[\mathbf B]_{kk}\ge0$, $k\ge 1$, under some constraints on $\rho_k$. 
We first provide two equivalent forms of $a_j^k$ according to \eqref{eq:aj}: $\forall 1\leq j\leq k-1$,
\begin{equation}\label{eq:akj_rem}
    \begin{aligned}
        a_{j}^k
        &=\int_0^1 \frac{-2 \tau_j(1-s)-\tau_{j+1}}{(\tau_{j}+\tau_{j+1})(t_k^*-(t_{j-1}+s\tau_j))^\alpha}\,\ds\\
        &=\frac{1}{\tau_{j}+\tau_{j+1}}\int_0^1 (t_k^*-(t_{j-1}+s \tau_j))^{-\alpha}\,{\rm d}( \tau_j s^2-(2\tau_j+\tau_{j+1})s)\\
     &=-(t_k^*-t_j)^{-\alpha}+\frac{\alpha\tau_j}{\tau_j+\tau_{j+1}}
    \int_0^{1} (\tau_j+\tau_{j+1}+s\tau_j)
    (1-s)(t_k^*-t_{j}+s\tau_j)^{-\alpha-1}\,\ds
    \end{aligned}
\end{equation}
and
\begin{equation}\label{eq:akj1_rem}
    \begin{aligned}
        a_{j}^k & =\int_0^1 \frac{-2 \tau_j(1-s)-\tau_{j+1}}{(\tau_{j}+\tau_{j+1})(t_k^*-(t_{j-1}+s\tau_j))^\alpha}\,\ds=\int_0^1 \frac{-2 \tau_j s-\tau_{j+1}}{(\tau_{j}+\tau_{j+1})(t_k^*-t_{j}+s \tau_j)^\alpha}\,\ds\\
        &=\frac{1}{\tau_{j}+\tau_{j+1}}\int_0^1 (t_k^*-t_{j}+s \tau_j)^{-\alpha}\,{\rm d}( - \tau_j s^2-\tau_{j+1}s)\\
    &=-(t_k^*-t_{j-1})^{-\alpha}-\frac{\alpha\tau_j}{\tau_j+\tau_{j+1}}
    \int_0^{1} (\tau_j+\tau_{j+1}-s\tau_j)(1-s)
    (t_k^*-t_{j-1}-s\tau_j)^{-\alpha-1}\,\ds.
    \end{aligned}
\end{equation}
Furthermore, we also reformulate $c_j^k$ in \eqref{eq:aj} as: $\forall 1\leq j\leq k-1$,
\begin{equation}\label{eq:ckj_rem}
\begin{aligned}
   c^{k}_j   & 
   = \int_0^1 \frac{\tau_j^2(2s-1)}{\tau_{j+1}(\tau_{j}+\tau_{j+1})(t_k^*-(t_{j-1}+s \tau_j))^\alpha}\,\ds\\
   &=\frac{\tau_j^2}{\tau_{j+1}(\tau_{j}+\tau_{j+1})}\int_0^1(t_k^*-(t_{j-1}+s \tau_j))^{-\alpha} {\rm d}(s^2-s)\\
  &=\frac{\alpha\tau_{j}^3}{\tau_{j+1}(\tau_{j}+\tau_{j+1})}\int_0^{1} s(1-s)(t_k^*-t_j+s\tau_{j})^{-\alpha-1}\  \ds.
  \end{aligned}
\end{equation}
In the following content, we consider four cases: $k=1$, $k=2$, $3\leq k\leq n-1$, and $k=n$.

{\bf Case 1:} When $k=1$, from \eqref{eq:aj} and $2\beta_1= -a^{2}_1$ in \eqref{betak}, we have
\begin{align*}
    [\mathbf B]_{11}&=\frac{\sigma^{1-\alpha}}{(1-\alpha)\tau_1^{\alpha}}-\frac12\int_0^1 \frac{2 \tau_1(1-\theta)+\tau_{2}}{(\tau_{1}+\tau_{2})(t_2^*-(t_{0}+\theta \tau_1))^\alpha}\,\dtheta\\
    &=\frac{\sigma^{1-\alpha}}{(1-\alpha)\tau_1^{\alpha}}-\frac{1}{2\tau_1^\alpha}\int_0^1\frac{2s+\rho_2}{(1+\rho_2)(\sigma\rho_2+s)^\alpha}\ds\\
    &> \frac{\sigma^{1-\alpha}}{(1-\alpha)\tau_1^{\alpha}}-\frac{1}{2\tau_1^\alpha(\sigma\rho_2)^\alpha}\int_0^1\frac{2s+\rho_2}{(1+\rho_2)}\ds=\frac{1}{2(1-\alpha)(\sigma\tau_1)^\alpha}\left(2\sigma-\frac{1-\alpha}{\rho_2^\alpha}\right).
    \end{align*}
 To ensure $ [\mathbf B]_{11}\ge 0$, we impose
\begin{equation}\label{eq:B1}
    2\sigma-\frac{1-\alpha}{\rho_2^\alpha}\ge0.
\end{equation}

{\bf Case 2:} When $k=2$, combining $2\beta_2 -d^{3}_2=a^{3}_1-a^{2}_1$ in \eqref{betak} and the property (P6) in Lemma \eqref{lemmapro}  gives
\begin{equation}\label{eq:b22}
    \begin{aligned}
      [\mathbf B]_{22} =&c_{1}^2+\frac{\sigma^{1-\alpha}}{(1-\alpha)\tau_2^{\alpha}}-\frac12(d^{3}_2+a^3_1-a^2_1)\\
    =&\frac12c_{1}^2+\frac{\sigma^{1-\alpha}}{(1-\alpha)\tau_2^{\alpha}}+\frac12(a^2_1-a^3_1+a^3_2)+\frac12(c^2_1-c^3_1)\\
    \ge&\frac12c_{1}^2+\frac{\sigma^{1-\alpha}}{(1-\alpha)\tau_2^{\alpha}}+\frac12(a^2_1-a^3_1+a^3_2).
    \end{aligned}
\end{equation}
Using the forms \eqref{eq:akj_rem} for $a_1^2,~a_1^3$ and \eqref{eq:akj1_rem} for $a_2^3$, we can derive
\begin{align}\label{a123}
    a^2_1-a^3_1+a^3_2=&-(\sigma\tau_2)^{-\alpha}+\frac{\alpha\tau_1}{\tau_1+\tau_2}
    \int_0^{1} (\tau_1+\tau_{2}+s\tau_1)
    (1-s)(t_2^*-t_{1}+s\tau_1)^{-\alpha-1}\,\ds\\\nonumber
    &-\frac{\alpha\tau_1}{\tau_1+\tau_{2}}
    \int_0^{1} (\tau_1+\tau_{2}+s\tau_1)
    (1-s)(t_3^*-t_{1}+s\tau_1)^{-\alpha-1}\,\ds
    \\\nonumber
    &-\frac{\alpha\tau_{2}}{\tau_{2}+\tau_{3}}
     \int_0^{1}(\tau_{2}+\tau_{3}-s\tau_{2}) (1-s)(t_3^*-t_{1}-s\tau_2)^{-\alpha-1}\,\ds\\\nonumber
     >&-(\sigma\tau_2)^{-\alpha}-\frac{\alpha\tau_{2}}{\tau_{2}+\tau_{3}}
     \int_0^{1}(\tau_{2}+\tau_{3}-s\tau_{2}) (1-s)(\tau_2+\sigma\tau_3-s\tau_2)^{-\alpha-1}\,\ds\\\nonumber
     =&-(\sigma\tau_2)^{-\alpha}-\frac{\alpha}{(1+\rho_{3})\tau_2^\alpha}
    \int_0^{1} 
    s(\rho_{3}+s)(\sigma\rho_{3}+s)^{-\alpha-1}\,\ds\\\nonumber
     >&-(\sigma\tau_2)^{-\alpha}-\frac{\alpha}{(1+\rho_{3})(\sigma\tau_2)^\alpha\rho_{3}^\alpha}
     \int_0^{1}\frac{s(\rho_{3}+s)}{\sigma\rho_{3}+s} \,\ds.\nonumber
\end{align}
Substituting \eqref{a123} into \eqref{eq:b22} yields
\begin{equation*}
\begin{aligned}
    [\mathbf B]_{22}
    \ge&\frac12c_{1}^2+\frac{1}{2(1-\alpha)(\sigma\tau_2)^{\alpha}}\left(2\sigma-(1-\alpha)-\frac{\alpha(1-\alpha)}{(1+\rho_{3})\rho_{3}^\alpha}
     \int_0^{1}\frac{s(\rho_{3}+s)}{\sigma\rho_{3}+s} \,\ds\right).
    \end{aligned}
\end{equation*}
To make sure $[\mathbf B]_{22}\ge0$, we impose
\begin{equation}\label{eq:B2}
    2\sigma-(1-\alpha)-\frac{\alpha(1-\alpha)}{(1+\rho_3)\rho_{3}^\alpha}
     \int_0^{1}\frac{s(\rho_{3}+s)}{\sigma\rho_{3}+s} \,\ds\ge0.
\end{equation}

{\bf Case 3:} When $3\le k \le n-1$,
using $2\beta_k=d^{k+1}_k+d^{k}_{k-1}-d^{k+1}_{k-1}$ in \eqref{betak} and $d^k_j=c^k_{j-1}-a^k_j$, we have
\begin{align}\label{eq:kn}
    & [\mathbf B]_{kk}=\frac{\sigma^{1-\alpha}}{(1-\alpha)\tau_k^{\alpha}}+\frac12c_{k-1}^k+\frac12(c_{k-1}^k-d^{k+1}_k-d^{k}_{k-1}+d^{k+1}_{k-1})\\
    & = \frac{\sigma^{1-\alpha}}{(1-\alpha)\tau_k^{\alpha}}+\frac12c_{k-1}^k+\frac12[(c^{k}_{k-1}-c^{k+1}_{k-1})-(c^{k}_{k-2}-c^{k+1}_{k-2} )+(-a_{k-1}^{k+1}+a^{k+1}_{k}+a^k_{k-1})].\nonumber
\end{align}
From  \eqref{eq:akj_rem} -- \eqref{eq:ckj_rem}, if \eqref{condition:rho} holds for $j=k-1$, we have
\begin{align}\label{eq:3.80-1}
    &(c^{k}_{k-1}-c^{k+1}_{k-1})-(c^{k}_{k-2}-c^{k+1}_{k-2} )+(-a_{k-1}^{k+1}+a^{k+1}_{k}+a^k_{k-1})\\
    =&\frac{\alpha\tau_{k-1}^3}{\tau_{k}(\tau_{k-1}+\tau_{k})}\int_0^{1} s(1-s)\bigg[(t_k^*-t_{k-1}+s\tau_{k-1})^{-\alpha-1}
     -(t_{k+1}^*-t_{k-1}+s\tau_{k-1})^{-\alpha-1}\bigg]\  \ds\nonumber\\
   &-\frac{\alpha\tau_{k-2}^3}{\tau_{k-1}(\tau_{k-2}+\tau_{k-1})}\int_0^{1} s(1-s)\bigg[(t_k^*-t_{k-2}+s\tau_{k-2})^{-\alpha-1}\nonumber\\
   & \quad -(t_{k+1}^*-t_{k-2}+s\tau_{k-2})^{-\alpha-1}\bigg]\  \ds\nonumber\\
   &+\frac{\alpha\tau_{k-1}}{\tau_{k-1}+\tau_{k}}
    \int_0^{1} (\tau_{k-1}+\tau_{k}+s\tau_{k-1})
    (1-s)\bigg[(t_k^*-t_{k-1}+s\tau_{k-1})^{-\alpha-1}\nonumber\\
    &\quad-(t_{k+1}^*-t_{k-1}+s\tau_{k-1})^{-\alpha-1}\bigg]\,\ds\nonumber\\
    &-(\sigma\tau_k)^{-\alpha}-\frac{\alpha\tau_{k}}{\tau_{k}+\tau_{k+1}}
     \int_0^{1}(\tau_{k}+\tau_{k+1}-s\tau_{k}) (1-s)(t_{k+1}^*-t_{k-1}-s\tau_{k})^{-\alpha-1}\,\ds\nonumber\\
   >&-(\sigma\tau_k)^{-\alpha}-\frac{\alpha\tau_{k}}{\tau_{k}+\tau_{k+1}}
     \int_0^{1}s(\tau_{k+1}+s\tau_{k}) (\sigma\tau_{k+1}+s\tau_{k})^{-\alpha-1}\,\ds\nonumber\\
     =&-(\sigma\tau_k)^{-\alpha}-\frac{\alpha}{(1+\rho_{k+1})\tau_k^\alpha}
    \int_0^{1} 
    s(\rho_{k+1}+s)(\sigma\rho_{k+1}+s)^{-\alpha-1}\,\ds\nonumber\\
     >&-(\sigma\tau_k)^{-\alpha}-\frac{\alpha}{(1+\rho_{k+1})(\sigma\tau_k)^\alpha\rho_{k+1}^\alpha}
     \int_0^{1}\frac{s(\rho_{k+1}+s)}{\sigma\rho_{k+1}+s} \,\ds\nonumber,
    \end{align}
   where we use the forms \eqref{eq:akj_rem} for $a_{k-1}^{k},~a_{k-1}^{k+1}$ and \eqref{eq:akj1_rem} for $a_k^{k+1}$. 
The first inequality in \eqref{eq:3.80-1} can be derived as follows.
For fixed $j$, it is easy to see that
\begin{equation*}
       (t_k^*-t_{k-1}+s\tau_{k-1})^{-\alpha-1}-(t_{k+1}^*-t_{k-1}+s\tau_{k-1})^{-\alpha-1}>0
\end{equation*}
decreases w.r.t. $s$ and
$
 \int_0^1(1-3s)(1-s)= 0,
$
thus
\begin{equation*}
\begin{aligned}
    &\int_0^{1} (\tau_{k-1}+\tau_{k}+s\tau_{k-1})
    (1-s)[(t_k^*-t_{k-1}+s\tau_{k-1})^{-\alpha-1}-(t_{k+1}^*-t_{k-1}+s\tau_{k-1})^{-\alpha-1}]\,\ds\\
    &\geq    \int_0^{1} (4\tau_{k-1}+3\tau_{k})s
    (1-s)[(t_k^*-t_{k-1}+s\tau_{k-1})^{-\alpha-1}-(t_{k+1}^*-t_{k-1}+s\tau_{k-1})^{-\alpha-1}]\,\ds.
\end{aligned}
\end{equation*}
 Moreover the convexity of the function $t^{-1-\alpha}$ gives
\begin{equation*}
\begin{aligned}
    &     (t_k^*-t_{k-1}+s\tau_{k-1})^{-\alpha-1}-(t_{k+1}^*-t_{k-1}+s\tau_{k-1})^{-\alpha-1}\\
    & >(t_k^*-t_{k-2}+s\tau_{k-2})^{-\alpha-1}-(t_{k+1}^*-t_{k-2}+s\tau_{k-2})^{-\alpha-1},
\end{aligned}
\end{equation*}
Then we can get  the following result:
\begin{align*}
    &\frac{\alpha\tau_{k-1}^3}{\tau_{k}(\tau_{k-1}+\tau_{k})}\int_0^{1} s(1-s)\bigg[(t_k^*-t_{k-1}+s\tau_{k-1})^{-\alpha-1}
     -(t_{k+1}^*-t_{k-1}+s\tau_{k-1})^{-\alpha-1}\bigg]\  \ds\\
   &-\frac{\alpha\tau_{k-2}^3}{\tau_{k-1}(\tau_{k-2}+\tau_{k-1})}\int_0^{1} s(1-s)\bigg[(t_k^*-t_{k-2}+s\tau_{k-2})^{-\alpha-1} \\
   &\quad -(t_{k+1}^*-t_{k-2}+s\tau_{k-2})^{-\alpha-1}\bigg]\  \ds\\
   &+\frac{\alpha\tau_{k-1}}{\tau_{k-1}+\tau_{k}}
    \int_0^{1} (\tau_{k-1}+\tau_{k}+s\tau_{k-1})
    (1-s)\bigg[(t_k^*-t_{k-1}+s\tau_{k-1})^{-\alpha-1}\\
    &\quad-(t_{k+1}^*-t_{k-1}+s\tau_{k-1})^{-\alpha-1}\bigg]\,\ds\\
    >&\alpha\left(\frac{\tau_{k-1}^3}{\tau_{k}(\tau_{k-1}+\tau_{k})}-\frac{\tau_{k-2}^3}{\tau_{k-1}(\tau_{k-2}+\tau_{k-1})}+\frac{(4\tau_{k-1}+3\tau_{k})\tau_{k-1}}{\tau_{k-1}+\tau_{k}}\right)\int_0^{1} s(1-s)
     \\&\qquad\bigg[(t_k^*-t_{k-1}+s\tau_{k-1})^{-\alpha-1}
     -(t_{k+1}^*-t_{k-1}+s\tau_{k-1})^{-\alpha-1}\bigg]\, \ds
     \ge 0,
     \end{align*}
     as \eqref{condition:rho} for $j=k-1$ gives
     \begin{equation*}
         \frac{\tau_{k-1}^3}{\tau_{k}(\tau_{k-1}+\tau_{k})}-\frac{\tau_{k-2}^3}{\tau_{k-1}(\tau_{k-2}+\tau_{k-1})}+\frac{(4\tau_{k-1}+3\tau_{k})\tau_{k-1}}{\tau_{k-1}+\tau_{k}}\ge 0.
     \end{equation*}
Combining \eqref{eq:3.80-1} with  \eqref{eq:kn} yields
\begin{equation*}
\begin{aligned}
    [\mathbf B]_{kk}
    \ge&\frac12c_{k-1}^k+\frac{1}{2(1-\alpha)(\sigma\tau_k)^{\alpha}}\left(2\sigma-(1-\alpha)-\frac{\alpha(1-\alpha)}{(1+\rho_{k+1})\rho_{k+1}^\alpha}
     \int_0^{1}\frac{s(\rho_{k+1}+s)}{\sigma\rho_{k+1}+s} \,\ds\right).
    \end{aligned}
\end{equation*}
Thus, to ensure
$
[\mathbf B]_{kk} \ge 0
$ 
for $3\le k \le n-1$,
it is sufficient to impose  
\begin{equation}\label{eq:B3}
\begin{aligned}
&\frac{1}{\rho_{k}}\ge\frac{1}{\rho_{k-1}^2(1+\rho_{k-1})}-3,\\
&2\sigma-(1-\alpha)-\frac{\alpha(1-\alpha)}{(1+\rho_{k+1})\rho_{k+1}^\alpha}\int_0^{1}\frac{s(\rho_{k+1}+s)}{\sigma\rho_{k+1}+s} \,\ds\ge0.
 \end{aligned}
\end{equation}

{\bf Case 4:}
When $k=n$, we show $[\mathbf B]_{nn}\ge 0$ under some constraints on $\rho_n$.
From \eqref{betak}, \eqref{eq:akj_rem} and  \eqref{eq:ckj_rem}, we can derive
    \begin{align}\label{ineq:rhon}
   & [\mathbf B]_{nn} \\
   = &c_{n-1}^n+\frac{\sigma^{1-\alpha}}{(1-\alpha)\tau_n^{\alpha}}-\frac12(c^n_{n-2}-a^n_{n-1})\nonumber\\
   =&\frac12c_{n-1}^n+\frac{\sigma^{1-\alpha}}{(1-\alpha)\tau_n^{\alpha}}+\frac12(c_{n-1}^n-c^n_{n-2}+a^n_{n-1})\nonumber\\
    =&\frac12c_{n-1}^n+\frac{\sigma^{1-\alpha}}{(1-\alpha)\tau_n^{\alpha}}
    +\frac12\bigg(\frac{\alpha\tau_{n-1}^3}{\tau_{n}(\tau_{n-1}+\tau_{n})}\int_0^{1} s(1-s)(t_n^*-t_{n-1}+s\tau_{n-1})^{-\alpha-1}\  \ds\nonumber\\
    &-\frac{\alpha\tau_{n-2}^3}{\tau_{n-1}(\tau_{n-2}+\tau_{n-1})}\int_0^{1} s(1-s)(t_n^*-t_{n-2}+s\tau_{n-2})^{-\alpha-1}\  \ds  -(\sigma\tau_n)^{-\alpha}\nonumber\\
    &+\frac{\alpha\tau_{n-1}}{\tau_{n-1}+\tau_{n}}
    \int_0^{1} (\tau_{n-1}+\tau_{n}+s\tau_{n-1})
    (1-s)(t_n^*-t_{n-1}+s\tau_{n-1})^{-\alpha-1}\,\ds\bigg)\nonumber\\
     >&\frac12c_{n-1}^n+\frac{1}{2(1-\alpha)(\sigma\tau_{n})^{\alpha}}\left(
    2\sigma
  -(1-\alpha)\right),\nonumber
    \end{align}
if \eqref{condition:rho} holds for $j=n-1$.
The proof of the last inequality in \eqref{ineq:rhon} is similar to the previous proof of \eqref{eq:3.80-1},
where we use the facts 
\begin{equation*}
\begin{aligned}
    &\int_0^{1} (\tau_{n-1}+\tau_{n}+s\tau_{n-1})
    (1-s)(t_n^*-t_{n-1}+s\tau_{n-1})^{-\alpha-1}\,\ds\\
    &\geq    \int_0^{1} (4\tau_{n-1}+3\tau_{n})s
    (1-s)(t_n^*-t_{n-1}+s\tau_{n-1})^{-\alpha-1}\,\ds,
\end{aligned}
\end{equation*}
 and
\begin{equation*}
\begin{aligned}
    (t_n^*-t_{n-1}+s\tau_{n-1})^{-\alpha-1}
     >(t_n^*-t_{n-2}+s\tau_{n-2})^{-\alpha-1}.
\end{aligned}
\end{equation*}
We omit the details here. 
To ensure 
$
[\mathbf B]_{nn}\ge 0,
$
it is sufficient to impose
\begin{equation}\label{eq:B4}
\begin{aligned}
&\frac{1}{\rho_{n}}\ge\frac{1}{\rho_{n-1}^2(1+\rho_{n-1})}-3,\quad 2\sigma
  -(1-\alpha)\geq 0.
 \end{aligned}
\end{equation}

Combining \eqref{eq:B1}, \eqref{eq:B2}, \eqref{eq:B3} and \eqref{eq:B4},  we can conclude that if  the condition \eqref{condition:rho} holds for $3\le k\le n$ and 
\begin{equation}\label{conditionB}
\begin{aligned}
    &2\sigma-\frac{1-\alpha}{\rho_2^\alpha}\ge0,\\
    &2\sigma-(1-\alpha)-\frac{\alpha(1-\alpha)}{(1+\rho_{k+1})\rho_{k+1}^\alpha}\int_0^{1}\frac{s(\rho_{k+1}+s)}{\sigma\rho_{k+1}+s} \,\ds\ge0,\quad 2\le k\le n-1,\\
    &2\sigma
  -(1-\alpha)\geq 0,
\end{aligned}
\end{equation}
then $[\mathbf B]_{kk}\ge 0$, $k\ge1$.
We have proved the following results:
\begin{itemize}
    \item Positive semidefiniteness of $\mathbf A+\mathbf A^{\rm T}$: \eqref{condition:rho} holds;
    \item Positive definiteness of $\mathbf B$:  \eqref{conditionB} holds and  \eqref{condition:rho} holds for $3\le k\le n$;
\end{itemize}
which ensure 
\begin{align*}
    \mathbf M+\mathbf M^{\rm T} &= (\mathbf A+\mathbf A^{\rm T})+2\mathbf B\geq 2\mathbf B
    \geq (1-\alpha)^{-1}{\rm diag}\left(g_1(\alpha),g_2(\alpha),\ldots,g_{n}(\alpha)\right)\geq 0,
\end{align*}
where $g_k(\alpha)$ is given in \eqref{eq:gk}.
In the following content, we just simplify the above constraints for the positive semidefiniteness of $\mathbf M+\mathbf M^{\rm T}$. 

The condition \eqref{condition:rho} actually says that $(\rho_j,\rho_{j+1})$ lies on the right-hand side of the blue solid curve in Figure~\ref{fig:rho}.
Let $\rho_*\approx 0.356341$ be the root of $\rho(1+\rho)=1-3\rho^2(1+\rho).$
It can be found that if $\rho_{j}\leq  \rho_*$ for some $j$, then $\rho_*\geq\rho_j\geq \rho_{j+1}\geq\rho_{j+2}\geq \ldots$ and $\tau_j$ will shrink to $0$ quickly as $j$ increases.
This doesn't make sense in practice. We shall impose 
$
    \rho_{j}>\rho_*,~ \forall j\geq 2.
$
As a consequence, we have the following constraints: for $j\geq 2$,
\begin{equation}\label{condrhorho}
\left\{
\begin{aligned}
     &\rho_*<\rho_{j+1} \leq  \frac{\rho_j^2(1+\rho_j)}{1-3\rho_j^2(1+\rho_j)},&&  \rho_*<\rho_j< \eta,\\
      &\rho_*<\rho_{j+1},&& \eta\le \rho_j,
\end{aligned}
\right.
\end{equation}
where 
$ \eta\approx0.475329$
be the unique positive root of 
$1-3\rho^2(1+\rho)=0.$

 \begin{figure}
    \centering
    \includegraphics[width=0.5\textwidth]{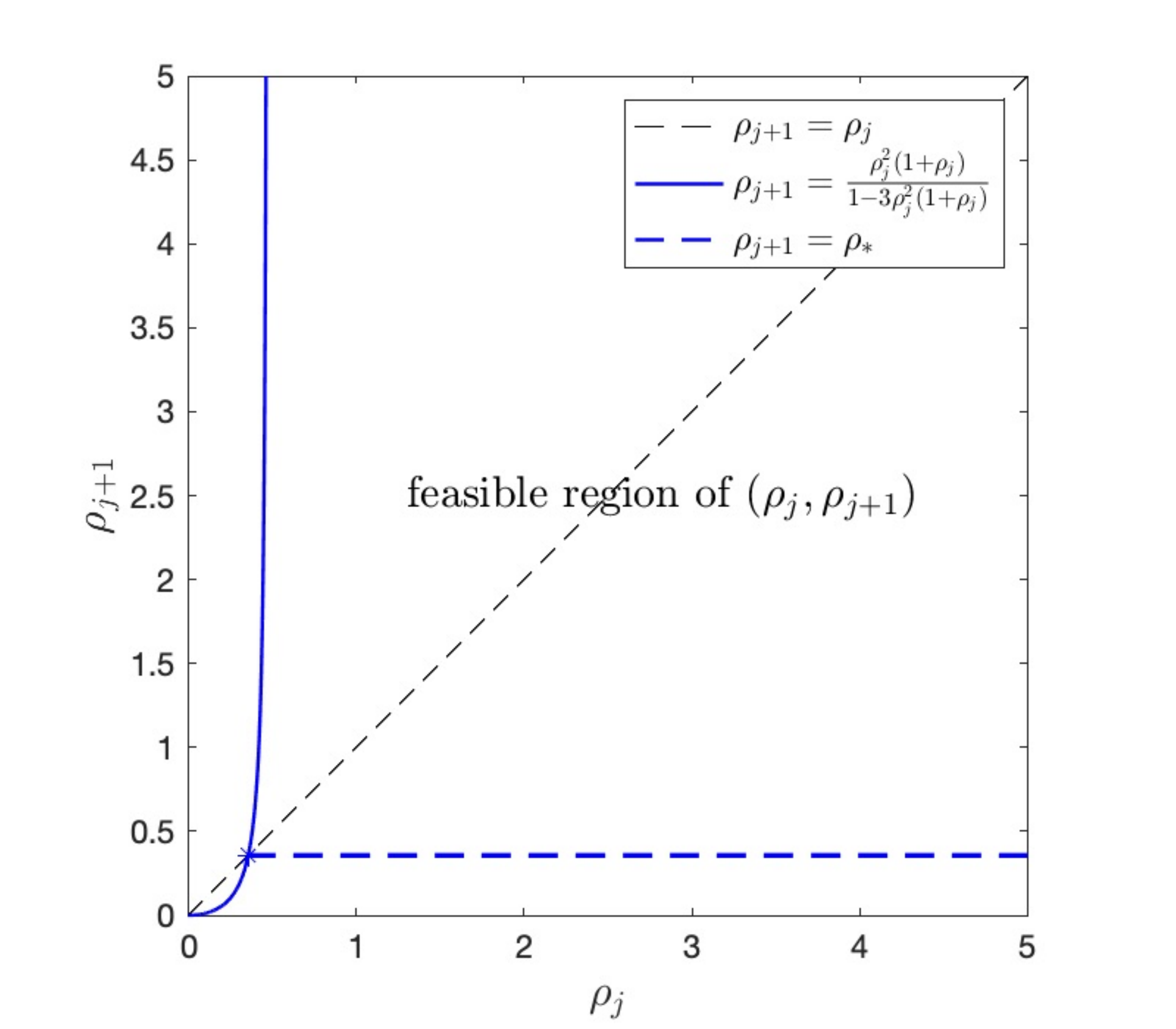}
    \caption{Feasible region of $(\rho_j,\rho_{j+1})$, on the right-hand side of the blue solid curve and above the blue dashed line, obtained from the constraint \eqref{condrhorho} for $j\geq 2$. The blue star marker denotes $(\rho_*,\rho_*)$.}
    \label{fig:rho}
\end{figure}

We now prove that \eqref{condrhorho} leads to \eqref{conditionB} when $\sigma = 1-\alpha/2\geq 1/2$.
In fact, it is easy to check that
\begin{equation*}
    2\sigma-\frac{1-\alpha}{\rho_2^\alpha}\ge2-\alpha-\frac{1-\alpha}{\rho_*^\alpha}\ge 0,\quad 2\sigma-(1-\alpha)=1,
\end{equation*}
and for $2\le k\le n-1$, we have 
\begin{align*}
       & 2\sigma-(1-\alpha)-\frac{\alpha(1-\alpha)}{(1+\rho_{k+1})\rho_{k+1}^\alpha}\int_0^{1}\frac{s(\rho_{k+1}+s)}{\sigma\rho_{k+1}+s} \,\ds\\\ge& 1-\frac{\alpha(1-\alpha)}{(1+\rho_{k+1})\rho_{k+1}^\alpha}\int_0^{1}\frac{s(\rho_{k+1}+s)}{\frac12\rho_{k+1}+s} \,\ds\\
        \ge& 1-\frac{\alpha(1-\alpha)}{(1+\rho_{k+1})\rho_{k+1}^\alpha}\ge 1-\frac{\alpha(1-\alpha)}{(1+\rho_*)\rho_*^\alpha}\ge 1-\frac{1}{4(1+\rho_*)\rho_*}\ge0.
\end{align*}

In summary, if \eqref{condrhorho} holds, then 
\begin{equation*}\label{eq:Ppd}
    \mathcal B_n(u,u) = \sum_{k=1}^{n}\langle L_k^\alpha u, \delta_k u\rangle\geq \sum_{k=1}^{n} \frac{g_k(\alpha)}{2\Gamma(2-\alpha)} \|\delta_k u\|^2_{L^2(\Omega)}\geq 0,
\end{equation*}
with $g_k(\alpha)$ given in \eqref{eq:gk}.
\end{proof}
\begin{remark}
If $\rho_k\ge \eta \approx0.475329$ for all $k\ge 2$, then the condition \eqref{thm1cond} holds, for which  the positive semidefiniteness of bilinear form $\mathcal B_n(u,u)$ \eqref{eq:PD} can be guaranteed.
\end{remark}

\section{Stability and convergence of L2-1$_\sigma$ method for subdiffusion equation}\label{sect4}
We consider the following subdiffusion equation:
\begin{equation}\label{eq:subdiffusion}
\begin{aligned}
    \partial_t^\alpha u(t,x) & = \Delta u(t,x)+f(t,x),&& (t,x)\in (0,\infty)\times\Omega,\\
    u(t,x) &= 0,&& (t,x)\in (0,\infty)\times\partial\Omega,\\
    u(0,x)& = u^0(x),&& x\in \Omega,
\end{aligned}    
\end{equation}
where $\Omega$ is a bounded Lipschitz domain in $\mathbb R^d$. 
Given an arbitrary nonuniform mesh $\{\tau_k\}_{k\ge1}$, the L2-$1_\sigma$ scheme of this subdiffusion equation is written as
\begin{equation}\label{eq:sch_sub}
\begin{aligned}
    L_k^{\alpha,*} u 
    &= (1-\alpha/2)\Delta u^k+\alpha/2\Delta u^{k-1}+f^k,&&  \text{in}~ \Omega,\\
    u^k&=0,&&\text{on} ~\partial\Omega,
\end{aligned}    
\end{equation}
where $f^k=f(t_k^*,\cdot)$.

\subsection{Global-in-time $H^1$-stability of L2-1$_\sigma$ scheme for subdiffusion equation}
\begin{theorem}\label{thm2}
Assume that $f(t,x) \in L^\infty([0,\infty);L^2(\Omega)) \cap BV([0,\infty); L^2(\Omega))$ is a bounded variation function in time and $ u^0\in H_0^1(\Omega)$.
If the nonuniform mesh $\{\tau_k\}_{k\ge1}$ satisfies \eqref{thm1cond} (for example $\rho_k\geq \eta\approx0.475329$ for $k\geq 2$),
then the numerical solution $u^n$ of the L2-1$_\sigma$ scheme \eqref{eq:sch_sub} satisfies the following global-in-time $H^1$-stability
\begin{equation*}
\begin{aligned}
    \|\nabla u^n\|_{L^2(\Omega)}
   &\le\|\nabla u^0\|_{L^2(\Omega)} +2C_{f}C_{\Omega},
\end{aligned}
\end{equation*}
where $C_f$ depends on the source term $f$, $C_\Omega$ is the Sobolev embedding constant depending on $\Omega$ and the spatial dimension $d$.
\end{theorem}
\begin{proof}
Multiplying \eqref{eq:sch_sub} with $\delta_k u$, integrating over $\Omega$, and summing up the derived equations over $k$ yield
\begin{equation*}
\begin{aligned}
\sum_{k=1}^{n}\langle L_k^\alpha u, \delta_k u\rangle
    =& \sum_{k=1}^{n}\langle (1-\alpha/2)\Delta u^k+\alpha/2\Delta u^{k-1}, \delta_k u\rangle +\sum_{k=1}^{n}\langle f^k, \delta_k u \rangle\\
     = & -\frac12 \|\nabla u^n\|_{L^2(\Omega)}^2 + \frac12 \|\nabla u^0\|_{L^2(\Omega)}^2 - \frac{1-\alpha}{2} \sum_{k=1}^{n} \|\nabla \delta_k u\|_{L^2(\Omega)}^2\\
     &+\langle f^n,  u^n\rangle-\langle f^1,  u^0\rangle-\sum_{k=2}^{n}\langle \delta_k f, u^{k-1}\rangle.\end{aligned}    
\end{equation*}
Applying the Cauchy--Schwarz inequality yields
\begin{equation*}
\begin{aligned}
    &\langle f^n,  u^n\rangle-\langle f^1,  u^0\rangle+\sum_{k=2}^{n}\langle \delta_k f,  u^{k-1}\rangle\\
    \le &\left( 2\|f\|_{L^\infty([0,\infty);L^2(\Omega))}+\|f\|_{BV([0,\infty); L^2(\Omega))}\right) \max_{0\le k\le n} {\|u^k\|_{L^2(\Omega)}}\\
     \le &C_{f}C_{\Omega}  \max_{0\le k\le n} {\|\nabla u^k\|_{L^2(\Omega)}},
\end{aligned} 
\end{equation*}
where $C_{f}= 2\|f\|_{L^\infty([0,\infty);L^2(\Omega))}+\|f\|_{BV([0,\infty); L^2(\Omega))}$,  and $C_{\Omega}$ is the Sobolev embedding constant depending on $\Omega$ and the spatial dimension. From Theorem~\ref{thm1}, we then have for $n\ge1$,
\begin{equation}\label{case2}
\begin{aligned}
\|\nabla u^n\|_{L^2(\Omega)}^2 
     \le &\|\nabla u^0\|_{L^2(\Omega)}^2 -  (1-\alpha)\sum_{k=1}^{n} \|\nabla \delta_k u\|_{L^2(\Omega)}^2-\sum_{k=1}^{n}\frac{g_k(\alpha)}{\Gamma(2-\alpha)}\| \delta_k u\|^2_{L^2(\Omega)}\\
     &+2C_{f} C_{\Omega} \max_{0\le k\le n} {\|\nabla u^k\|_{L^2(\Omega)}}\\
     \le &\|\nabla u^0\|_{L^2(\Omega)}^2+2C_{f}C_{\Omega}  \max_{0\le k\le n} {\|\nabla u^k\|_{L^2(\Omega)}}.
     \end{aligned}    
\end{equation}
For any $N \ge 1$, we take $\max_{0\le n\le N}$ on both sides of \eqref{case2}, to obtain
\begin{equation*}
\begin{aligned}
\max_{0\le n\le N}\|\nabla u^n\|_{L^2(\Omega)}^2 
     \le  \|\nabla u^0\|_{L^2(\Omega)}^2+2C_{f}C_{\Omega}  \max_{0\le n\le N} {\|\nabla u^n\|_{L^2(\Omega)}},
     \end{aligned}    
\end{equation*}
which indicates 
\begin{equation*}
\begin{aligned}
    \max_{0\le n\le N}\|\nabla u^n\|_{L^2(\Omega)}&\le C_{f}C_{\Omega} +\sqrt{(C_{f}C_{\Omega})^2+\|\nabla u^0\|_{L^2(\Omega)}^2}
   \le\|\nabla u^0\|_{L^2(\Omega)} +2C_{f}C_{\Omega}.
\end{aligned}
\end{equation*}
The proof is completed.
\end{proof}
\begin{remark}
Assume that the solution of subdiffusion equation satisfies $u(t,x)\in C([0,\infty);H_0^1(\Omega))\cap C^1((0,\infty);H_0^1(\Omega))$ and the source term satisfies
$f(t,x) \in C([0,\infty);L^2(\Omega)),~ \partial_t f(t,x)\in L^1([0,\infty);L^2(\Omega))$.
For any fixed $T>0$,
multiplying the first equation of \eqref{eq:subdiffusion} with $\partial_t u(t,x)$ and  integrating over $(0,T)\times\Omega$ yield
\begin{equation*}
\begin{aligned}
    &\int_0^T\int_\Omega \partial_t^\alpha u(t,x) \partial_t u(t,x)\ \dx\dt \\
    = & \frac12\int _0^T \int_\Omega \partial_t|\nabla u(t,x)|^2\ \dx\dt +\int_0^T\int_\Omega f(t,x)\partial_t u(t,x)\ \dx\dt.
\end{aligned}    
\end{equation*}
According to~\cite{tang2019energy},  
\begin{equation*}
    \int_0^T\int_\Omega \partial_t^\alpha u(t,x) \partial_t u(t,x)\ \dx\dt\ge 0,
\end{equation*}
and moreover,
\begin{align*}
    &\int_0^T\int_\Omega f(t,x) \partial_t u(t,x)\ \dx\dt\\=&\left(\int_\Omega f(t,x)u(t,x)\ \dx\right)\bigg|_0^T-\int_0^T\int_\Omega \partial_t f(t,x)u(t,x)\ \dx\dt\\
    \le& \left(2\|f\|_{L^\infty([0,\infty);L^2(\Omega))}+\int_0^\infty\| \partial_t f(t,x)\|_{L^2(\Omega)}\,\dt\right)C_\Omega\|\nabla u\|_{L^\infty([0,T];L^2(\Omega))}\\
    =:& C_f^{cont}C_\Omega\|\nabla u\|_{L^\infty([0,T];L^2(\Omega))}.
\end{align*}
Thus we derive the $H^1$-stability at the continuous level 
\begin{equation*}
\begin{aligned}
    \|\nabla u(T,x)\|_{L^2(\Omega)}\le 
   \|\nabla u(0,x)\|_{L^2(\Omega)}+2C_f^{cont}C_\Omega, \quad \forall\  T>0,
\end{aligned}
\end{equation*}
which corresponds to our $H^1$-stability result in Theorem~\ref{thm2} for the L2-1$_\sigma$ scheme of the subdiffusion equation \eqref{eq:subdiffusion}.
\end{remark}

\begin{remark}
In the case of $\alpha =1$, i.e., the standard diffusion equation, the energy stability (or $H^1$-stability) has been established for the second order BDF2 schemes in \cite[Theorem 2.1]{liao2021analysis} and for the third order BDF3 schemes in \cite[Theorem 3.1]{liao2022discrete} on general nonuniform meshes. 
\end{remark}

\subsection{Sharp convergence of L2-1$_\sigma$ scheme for subdiffusion equation} 
We show the error estimate of the L2-1$_\sigma$ scheme \eqref{eq:sch_sub} for the subdiffusion equation \eqref{eq:subdiffusion}, that is different from the one in~\cite{liao2018second,liao2019discrete}.
To be precise we will reduce the restriction on time step ratios from $\rho_k\geq 4/7$ in~\cite{liao2018second} to $\rho_k\geq 0.475329$. 
We first reformulate the discrete fractional operator \eqref{eq:Lk1}:
\begin{equation*}
    L^{\alpha,*}_k u=\frac{1}{\Gamma(1-\alpha)}\left( [\mathbf M]_{k,k} u^k - \sum_{j=2}^k([\mathbf M]_{k,j}-[\mathbf M]_{k,j-1}) u^{j-1}-[\mathbf M]_{k,1}u^0\right),
\end{equation*}
where $\mathbf M$ is given by \eqref{matM}.
We now give some properties on $[\mathbf M]_{k,j}$.
\begin{lemma}\label{proMkj}
Under the condition \eqref{thm1cond}, the following properties of $[\mathbf M]_{k,j}$ given by \eqref{matM} hold:
\begin{itemize}
\item[{\rm (Q1)}] 
\begin{equation}\label{eq:4.9}
[\mathbf M]_{k,j}\ge \frac{\rho_*}{(1+\rho_*) \tau_j}\int_{t_{j-1}}^{\min\{t_j,t_k^*\}}(t_k^*-s)^{-\alpha}\,\ds,\quad 1\le j\le k.
\end{equation}
\item[{\rm (Q2)}] 
For all $2\le j\le k-1$, 
\begin{equation*}
   [\mathbf M]_{k,j}-[\mathbf M]_{k,j-1}\ge \frac{\alpha\tau_j}{\tau_j+\tau_{j+1}}
    \int_0^{1} (\tau_j+\tau_{j+1}-s\tau_j)(1-s)
    (t_k^*-t_{j-1}-s\tau_j)^{-\alpha-1}\,\ds,
\end{equation*}
and
\begin{equation*}
    [\mathbf M]_{k,k}-[\mathbf M]_{k,k-1}\ge\frac{\alpha}{2(1-\alpha)(\sigma\tau_{k})^{\alpha}}.
\end{equation*}
\item[{\rm (Q3)}]  Moreover, if $\rho_k\geq\eta\approx0.475329$ for all $k\geq 2$, then
\begin{equation*}
    \frac{1-\alpha}{\sigma}[\mathbf M]_{k,k}-[\mathbf M]_{k,k-1}\ge 0.
\end{equation*}
Here $\eta$  is the real root of $1-3\rho^2(1+\rho) = 0$.
\end{itemize}
\end{lemma}
\begin{proof}  
From \eqref{matM}, for $1\le j\le k-1$, 
\begin{equation}\label{ineq:mrho2}
     \begin{aligned}
      [\mathbf M]_{k,j}&\ge- a_j^n=\int_0^1 \frac{2 \tau_j(1-\theta)+\tau_{j+1}}{(\tau_{j}+\tau_{j+1})(t_k^*-(t_{j-1}+\theta \tau_j))^\alpha}\,\dtheta\\
      &\ge \frac{\rho_{j+1}}{1+\rho_{j+1}}\int_0^1 \frac{1}{(t_k^*-(t_{j-1}+\theta \tau_j))^\alpha}\,\dtheta
      \ge \frac{\rho_*}{(1+\rho_*)\tau_j}\int_{t_{j-1}}^{t_j}(t_k^*-s)^{-\alpha}\ds,
     \end{aligned}
\end{equation}
     and for $j=k$, 
     \begin{equation*}
          [\mathbf M]_{k,k}=c^k_{k-1}+\frac{\sigma^{1-\alpha}}{(1-\alpha)\tau_k^\alpha}\ge\frac{\sigma^{1-\alpha}}{(1-\alpha)\tau_k^\alpha}=\frac{1}{\tau_k}\int_{t_{k-1}}^{t_k^*}(t_k^*-s)^{-\alpha}\ds.
     \end{equation*}
     The inequality \eqref{eq:4.9} holds.

For $ 2\le j\le k-1$, according to \eqref{eq:akj_rem} -- \eqref{eq:ckj_rem},
\begin{align*}
   &[\mathbf M]_{k,j}-[\mathbf M]_{k,j-1}\\
   =&\frac{\alpha\tau_{j-1}^3}{\tau_{j}(\tau_{j-1}+\tau_{j})}\int_0^{1} s(1-s)(t_k^*-t_{j-1}+s\tau_{j-1})^{-\alpha-1}\  \ds\\
    &-\frac{\alpha\tau_{j-2}^3}{\tau_{j-1}(\tau_{j-2}+\tau_{j-1})}\int_0^{1} s(1-s)(t_k^*-t_{j-2}+s\tau_{j-2})^{-\alpha-1}\  \ds\\
    &+\frac{\alpha\tau_{j-1}}{\tau_{j-1}+\tau_{j}}
    \int_0^{1} (\tau_{j-1}+\tau_{j}+s\tau_{j-1})
    (1-s)(t_k^*-t_{j-1}+s\tau_{j-1})^{-\alpha-1}\,\ds
    \\
    &+\frac{\alpha\tau_{j}}{\tau_{j}+\tau_{j+1}}
     \int_0^{1}(\tau_{j}+\tau_{j+1}-s\tau_{j}) (1-s)(t_k^*-t_{j-1}-s\tau_{j})^{-\alpha-1}\,\ds\\
\ge&
   \frac{\alpha\tau_j}{\tau_j+\tau_{j+1}}
    \int_0^{1} (\tau_j+\tau_{j+1}-s\tau_j)(1-s)
    (t_k^*-t_{j-1}-s\tau_j)^{-\alpha-1}\,\ds,
\end{align*}
under the condition \eqref{thm1cond} (for simplicity we make a convention that $\tau_0 =0$).
Note that \eqref{thm1cond} indicates the sum of first three terms is positive, using the techniques in \eqref{ineq:rhon}.
When $j=k=2$, we obtain from \eqref{eq:akj_rem}
    \begin{equation*}\label{bkkvere0}
        [\mathbf M]_{2,2}-[\mathbf M]_{2,1}=c_1^2+\frac{\sigma^{1-\alpha}}{(1-\alpha)\tau_2^{\alpha}}+a_1^2
        \ge \frac{\sigma^{1-\alpha}}{(1-\alpha)\tau_2^{\alpha}} - \frac{1}{(\sigma\tau_2)^\alpha}
        =\frac{\alpha}{2(1-\alpha)(\sigma\tau_{2})^{\alpha}},
    \end{equation*}
    where we use the fact $\sigma = 1-\alpha/2$.
 Moreover when $j=k\geq 3$, we have
    \begin{align*}\label{bkkvere}
  & [\mathbf M]_{k,k}-[\mathbf M]_{k,k-1}
    \\=&\frac{\sigma^{1-\alpha}}{(1-\alpha)\tau_k^{\alpha}}+(c_{k-1}^k-c^k_{k-2}+a^k_{k-1})\nonumber\\
    =&\frac{\sigma^{1-\alpha}}{(1-\alpha)\tau_k^{\alpha}}
    +\bigg(\frac{\alpha\tau_{k-1}^3}{\tau_{k}(\tau_{k-1}+\tau_{k})}\int_0^{1} s(1-s)(t_k^*-t_{k-1}+s\tau_{k-1})^{-\alpha-1}\  \ds\nonumber\\
    &-\frac{\alpha\tau_{k-2}^3}{\tau_{k-1}(\tau_{k-2}+\tau_{k-1})}\int_0^{1} s(1-s)(t_k^*-t_{k-2}+s\tau_{k-2})^{-\alpha-1}\  \ds\nonumber\\
    & -(\sigma\tau_k)^{-\alpha}+\frac{\alpha\tau_{k-1}}{\tau_{k-1}+\tau_{k}}
    \int_0^{1} (\tau_{k-1}+\tau_{k}+s\tau_{k-1})
    (1-s)(t_k^*-t_{k-1}+s\tau_{k-1})^{-\alpha-1}\,\ds\bigg)\nonumber\\
     >&\frac{\sigma^{1-\alpha}}{(1-\alpha)\tau_k^{\alpha}} - \frac{1}{(\sigma\tau_k)^\alpha}= \frac{\alpha}{2(1-\alpha)(\sigma\tau_{k})^{\alpha}},\nonumber
    \end{align*}
    when the condition \eqref{thm1cond} holds.
    This inequality coincide with \eqref{ineq:rhon} by replacing $n$ with $k$.
    
    For the property (Q3),  the case of $k=2$ is not difficult to obtain. In the case of $k\ge 3$, we have 
 \begin{align*}
  &  \frac{1-\alpha}{\sigma}[\mathbf M]_{k,k}-[\mathbf M]_{k,k-1}
    \\\ge&(\sigma\tau_k)^{-\alpha}-c^k_{k-2}+a^k_{k-1} \\
   =&(\sigma\tau_k)^{-\alpha}
    -\frac{\alpha\tau_{k-2}^3}{\tau_{k-1}(\tau_{k-2}+\tau_{k-1})}\int_0^{1} s(1-s)(t_k^*-t_{k-2}+s\tau_{k-2})^{-\alpha-1}\  \ds\\
    & -(\sigma\tau_k)^{-\alpha}+\frac{\alpha\tau_{k-1}}{\tau_{k-1}+\tau_{k}}
    \int_0^{1} (\tau_{k-1}+\tau_{k}+s\tau_{k-1})
    (1-s)(t_k^*-t_{k-1}+s\tau_{k-1})^{-\alpha-1}\,\ds\\
     >&\alpha\left(\frac{\tau_{k-1}(4\tau_{k-1}+3\tau_{k})}{\tau_{k-1}+\tau_{k}}-\frac{\tau_{k-2}^3}{\tau_{k-1}(\tau_{k-2}+\tau_{k-1})}\right)\int_0^{1} s(1-s)(t_k^*-t_{k-1}+s\tau_{k-1})^{-\alpha-1}\  \ds\\
     \ge&0, 
    \end{align*}
where we use the facts 
\begin{align*}
    &\int_0^{1} (\tau_{k-1}+\tau_{k}+s\tau_{k-1})
    (1-s)(t_k^*-t_{k-1}+s\tau_{k-1})^{-\alpha-1}\,\ds\\\ge& (4\tau_{k-1}+3\tau_{k})\int_0^{1} s(1-s)(t_k^*-t_{k-1}+s\tau_{k-1})^{-\alpha-1}\  \ds,\\
    &(t_k^*-t_{k-1}+s\tau_{k-1})^{-\alpha-1}\ge (t_k^*-t_{k-2}+s\tau_{k-2})^{-\alpha-1},
\end{align*}
and 
\begin{equation*}
    \frac{\tau_{k-1}(4\tau_{k-1}+3\tau_{k})}{\tau_{k-1}+\tau_{k}}-\frac{\tau_{k-2}^3}{\tau_{k-1}(\tau_{k-2}+\tau_{k-1})}\ge0, 
\end{equation*}
when $\rho_k\geq \eta\approx0.475329$ for all $k\geq 2$.

\end{proof}

 Consider the following three standard Lagrange interpolation operators with the following interpolation  points:
 \begin{equation*}
     \Pi_{1,j}:{t_{j-1},t_j},\quad \Pi_{2,j}:{t_{j-1},t_j,t_{j+1}},\quad \Pi^*_{2,j}:{t_{j-1},t^*_j,t_j}.
\end{equation*}
As stated in~\cite{kopteva2020error}, when $\sigma=1-\alpha/2$, 
\begin{equation*}\label{eq:truncneed}
    \int_{t_k-1}^{t_k^*}(\Pi_{1,k}v-\Pi^*_{2,k}v)'(s)(t_k^*-s)^{-\alpha}\,\ds=0.
\end{equation*}
We now analyze the approximation error of the discrete fractional operator in the following lemma.
\begin{lemma}\label{lem4.3}
Given a function $u$ satisfying $|\partial_t^{m} u(t)|\le C_m(1+t^{\alpha-m})$ for $m=1,3$ and nonuniform mesh $\{\tau_k\}_{k\ge 1}$ satisfying condition \eqref{thm1cond}, the approximation error is given by 
\begin{equation}\label{eq:localrk}
r_k \coloneqq \frac{1}{\Gamma(1-\alpha)}\int _0^{t_k^*}(t_k^*-s)^{-\alpha}\partial_s[u(s)-I_2u(s)]\,\ds,\quad k\ge 1,
\end{equation} 
where $I_2 u=\Pi_{2,j}u$ on $(t_{j-1},t_j)$ for $j<k$ and $I_2 u=\Pi^*_{2,k}u$ on $(t_{k-1},t_k^*)$.
Then for $k\geq 1$,
\begin{equation}\label{eq:approxerror}
\begin{aligned}
   | r_k|\le \frac{C}{\Gamma(1-\alpha)}\bigg(  [\mathbf M]_{k,1}\big(t_2^\alpha/\alpha+t_2\big)+\sum_{j=2}^k([\mathbf M]_{k,j}-[\mathbf M]_{k,j-1})(1+\rho_{j+1})(1+t_{j-1}^{\alpha-3})\tau_j^3\bigg),
\end{aligned}
\end{equation}
where $C$ is a constant depending on $C_m$ for $m=1,3$ and $\rho_{k+1}=1$.
\end{lemma}
 \begin{proof}
The case of $k=1$ is not difficult to prove. We now consider the case of $k\geq 2$.
 Let $\chi(s) \coloneqq u-I_2u$. 
 Three subcases are discussed in the following content. 
 
{\bf Subcase 1.} On the interval $(t_0,t_1)$, we have
 \begin{equation*}
    \partial_s I_2u (s) =  \frac{2s -t_1-t_{2}}{\tau_{1}(\tau_{1}+\tau_{2})} u(t_0) -  \frac{2s -t_{2}}{\tau_{1}\tau_{2} }u(t_1)
     + \frac{2s -t_{1}}{\tau_{2}(\tau_{1}+\tau_{2})} u(t_2)
 \end{equation*}
that is linear w.r.t. $s$.
Then we have
 \begin{align*}
     |\partial_s I_2u (s)|& \leq \max\{|\partial_s I_2u (t_0)|,|\partial_s I_2u (t_1)|\}
      \leq C_1\frac{ 1+\rho_2}{\tau_1\rho_{2}}(t_2+t_2^\alpha/\alpha),
     \end{align*}
 where we use the facts
  \begin{align*}
       \partial_s I_2u (t_0) &= - \frac{ 2\tau_1+\tau_2}{\tau_{1}(\tau_{1}+\tau_{2})} u(t_0) +  \frac{ \tau_1+\tau_2}{\tau_{1}\tau_{2} }u(t_1)
     - \frac{\tau_1}{\tau_{2}(\tau_{1}+\tau_{2})} u(t_2) \\
    &=-\frac{ 2\tau_1+\tau_2}{\tau_{1}(\tau_{1}+\tau_{2})} (u(t_0)-u(t_1) )
     +\frac{\tau_1}{\tau_{2}(\tau_{1}+\tau_{2})} (u(t_1)-u(t_2)) \\
       &\le\left(\frac{ 2\tau_1+\tau_2}{\tau_{1}(\tau_{1}+\tau_{2})}+\frac{\tau_1}{\tau_{2}(\tau_{1}+\tau_{2})} \right) \max\{|u(t_0)-u(t_1)|,|u(t_1)-u(t_2)|\}
      \\
      &=\frac{ \tau_1+\tau_2}{\tau_{1}\tau_{2}} \max\{|u(t_0)-u(t_1)|,|u(t_1)-u(t_2)|\},
      \\
     \partial_s I_2u (t_1)&=-\frac{\tau_2}{\tau_{1}(\tau_{1}+\tau_{2})} u(t_0) -  \frac{\tau_1-\tau_2}{\tau_{1}\tau_{2} }u(t_1)
     + \frac{\tau_1}{\tau_{2}(\tau_{1}+\tau_{2})} u(t_2)\\
     &=-\frac{\tau_2}{\tau_{1}(\tau_{1}+\tau_{2})} (u(t_0) -u(t_1))
     - \frac{\tau_1}{\tau_{2}(\tau_{1}+\tau_{2})}( u(t_1)-u(t_2))\\
     &\le \left(\frac{\tau_2}{\tau_{1}(\tau_{1}+\tau_{2})}+\frac{\tau_1}{\tau_{2}(\tau_{1}+\tau_{2})}\right)\max\{|u(t_0)-u(t_1)|,|u(t_1)-u(t_2)|\}
     \\& =\frac{\tau_1^2+\tau_2^2}{\tau_{1}\tau_2(\tau_{1}+\tau_{2})}\max\{|u(t_0)-u(t_1)|,|u(t_1)-u(t_2)|\},\\
      |u(t_0)-u(t_1)|&=|\int_0^{t_1}\partial_s u(s)\,\ds|\le C_1(\tau_1+\tau_1^\alpha/\alpha),\\
      |u(t_1)-u(t_2)|&=|\int_{t_1}^{t_2}\partial_s u(s)\,\ds|\le C_1(\tau_2+(t_2^\alpha-t_1^\alpha)/\alpha).
\end{align*}
Therefore, we have
  \begin{equation*}
     |\partial_s\chi(s)|\le |\partial_s u|+|\partial_s I_2u|
     \leq  C_1\left(s^{\alpha-1}+ 1+\frac{ 1+\rho_2}{\tau_1\rho_{2}}(t_2+t_2^\alpha/\alpha)\right),
 \end{equation*}
which yields
 \begin{equation}\label{eq:intervalt1}
    \begin{aligned}
  &| \frac{1}{\Gamma(1-\alpha)}\int _0^{t_1}(t^*_k-s)^{-\alpha}\partial_s\chi(s)\,\ds|\\
  \le&   \frac{C_1}{\Gamma(1-\alpha)}\left(\int _0^{t_1}s^{\alpha-1}(t_k^*-s)^{-\alpha}\,\ds+\frac{\tau_1+(1+\rho_2)/\rho_2(t_2+t_2^\alpha/\alpha)}{\tau_1}\int _0^{t_1}(t_k^*-s)^{-\alpha}\,\ds\right)\\
  \le& \frac{C_1}{\Gamma(1-\alpha)}\left(\frac{\tau_1^\alpha}{\alpha(t_k^*-\tau_1)^\alpha}+\frac{\tau_1+(1+\rho_2)/\rho_2(t_2+t_2^\alpha/\alpha)}{\tau_1}\int _0^{t_1}(t_k^*-s)^{-\alpha}\,\ds\right)\\
  \le&\frac{C(t_2^\alpha/\alpha+t_2)}{\Gamma(1-\alpha)}[\mathbf M]_{k,1},
    \end{aligned}
\end{equation}
where $C$ is an  absolute constant only depending on $C_1$.
In the last inequality of \eqref{eq:intervalt1}, we use the fact
\begin{align*}
    [\mathbf M]_{k,1}&\ge\frac{\rho_2}{(1+\rho_2)\tau_1}\int _0^{t_1}(t_k^*-s)^{-\alpha}\,\ds\ge \frac{\rho_2}{(1+\rho_2)(t_k^*)^\alpha}\\&\ge\frac{\rho_2^{1+\alpha}}{(1+\rho_2)(2+\rho_2)^\alpha(t_k^*-\tau_1)^\alpha}\ge\frac{\rho_*^{1+\alpha}}{(1+\rho_*)(2+\rho_*)^\alpha(t_k^*-\tau_1)^\alpha}
\end{align*}
obtained from the inequality \eqref{ineq:mrho2}.

{\bf Subcase 2.}
On the interval $(t_{j-1},t_j)$, $2\le j\le k-1$,
\begin{equation*}
\begin{aligned}
    |\chi(s)|&=|\frac{u^{(3)}(\xi)}{6}(s-t_{j-1})(s-t_{j})(s-t_{j+1})|\
    \le C_3(1+t_{j-1}^{\alpha-3})(s-t_{j-1})(s-t_{j})(s-t_{j+1}),
\end{aligned}
\end{equation*}
where $\xi \in (t_{j-1},t_{j+1})$. 
Then we have 
\begin{align}\label{eq:intervaltj}
&| \frac{1}{\Gamma(1-\alpha)}\int _{t_{j-1}}^{t_j}(t_k^*-s)^{-\alpha}\partial_s\chi(s)\,\ds|=
|\frac{-\alpha}{\Gamma(1-\alpha)}\int _{t_{j-1}}^{t_j}(t_k^*-s)^{-\alpha-1}\chi(s)\,\ds|\\\nonumber
\le&\frac{C_3\alpha (1+t_{j-1}^{\alpha-3})}{\Gamma(1-\alpha)}\int _{t_{j-1}}^{t_j}(t_k^*-s)^{-\alpha-1}(s-t_{j-1})(s-t_{j})(s-t_{j+1})\,\ds  \\\nonumber
=&\frac{C_3\alpha  (1+t_{j-1}^{\alpha-3})\tau^3_j}{\Gamma(1-\alpha)}\int_0^1s(\tau_j+\tau_{j+1}-s\tau_j)(1-s)(t_k^*-t_{j-1}-s\tau_j)^{-\alpha-1}\,\ds\\\nonumber
\le &\frac{C_3(1+\rho_{j+1}) (1+t_{j-1}^{\alpha-3})\tau^3_j}{\Gamma(1-\alpha)} ([\mathbf M]_{k,j}-[\mathbf M]_{k,j-1}),\nonumber
\end{align}
from (Q2) in Lemma~\ref{proMkj}.

{\bf Subcase 3.}
On the interval $(t_{k-1},t_k^*)$, 
\begin{align*}
      |\chi(s)|\le& C_3(1+t_{k-1}^{\alpha-3})(s-t_{k-1})(t_k^*-s) (t_k-s)
    \le C_3(1+t_{k-1}^{\alpha-3})\tau_k^2(t_k^*-s), 
\end{align*}
which yields
\begin{equation}\label{eq:intervaltk}
\begin{aligned}
&| \frac{1}{\Gamma(1-\alpha)}\int _{t_{k-1}}^{t_k^*}(t_k^*-s)^{-\alpha}\partial_s\chi(s)\,\ds|=|\frac{-\alpha}{\Gamma(1-\alpha)}\int _{t_{k-1}}^{t_k^*}(t_k^*-s)^{-\alpha-1}\chi(s)\,\ds|\\
\le&\frac{C_3\alpha(1+ t_{k-1}^{\alpha-3})\tau_k^2}{\Gamma(1-\alpha)}\int _{t_{k-1}}^{t_k^*}(t_k^*-s)^{-\alpha}\,\ds  =\frac{2C_3\sigma(1+ t_{k-1}^{\alpha-3})\tau^3_k }{\Gamma(1-\alpha)} \frac{\alpha}{2(1-\alpha)(\sigma\tau_k)^\alpha}\\
\le &\frac{2 C_3\sigma (1+t_{k-1}^{\alpha-3})\tau^3_k}{\Gamma(1-\alpha)} ([\mathbf M]_{k,k}-[\mathbf M]_{k,k-1})
\end{aligned}
\end{equation}
from (Q2) in Lemma~\ref{proMkj}.

Combining \eqref{eq:intervalt1}, \eqref{eq:intervaltj} and \eqref{eq:intervaltk} we obtain the estimation \eqref{eq:approxerror} of approximation error. 
 \end{proof}

\begin{theorem}\label{thm:err}
Assume that $u\in C^3((0,T],H^1_0(\Omega))$ and $|\partial_t^{m} u(t)|\le C_m(1+t^{\alpha-m})$, for $m=1,2,3$ for $0< t\le T$. If the nonuniform mesh satisfies $\rho_k\geq \eta\approx0.475329$, then the numerical solutions of L2-1$_\sigma$ scheme \eqref{eq:sch_sub} have the following global  error estimate
\begin{align*}
  &\max_{1\le k\le n}  \|u(t_k)-u^k\|_{L^2(\Omega)}\\\le  &C\bigg(t_2^\alpha/\alpha+t_2+\frac{1}{1-\alpha}\max_{2\le k\le n}(1+\rho_{k+1}) (1+t_{k-1}^{\alpha-3})(t_{k-1}^*)^\alpha\tau_k^3\tau_{k-1}^{-\alpha}\\
  &+(\tau_1^\alpha/\alpha+\tau_1)\tau_1^{\alpha/2}+\sqrt{\Gamma(1-\alpha)}\max_{2\le k\le n}(t_k^*)^{\alpha/2}(1+ t_{k-1}^{\alpha-2})\tau_k^2\bigg),
\end{align*}
where $C$ is a constant depending only on $C_m$, $m=1,2,3$ and $\Omega$.
\end{theorem}

\begin{proof} Let  $e^k\coloneqq
u(t_k)-u^k$. We have
\begin{equation}\label{eq:errorfun1}
     L_k^{\alpha,*} e 
    = \Delta e_k^*-r_k+\Delta R_k^*,
\end{equation}
where $ e_k^*\coloneqq (1-\alpha/2)e^k+\alpha/2 e^{k-1}$, $r_k$ is given in  \eqref{eq:localrk}, and $R_k^* \coloneqq u(t_k^*)-((1-\alpha/2)u(t_k)+\alpha/2 u(t_{k-1}))$.
Multiplying \eqref{eq:errorfun1} with $e_k^*$ and integrating over $\Omega$ yield 
\begin{equation}\label{eq:errorfun}
\begin{aligned}
    \langle L_k^{\alpha,*} e ,e_k^*\rangle
    &= -\|\nabla e_k^*\|_{L^2(\Omega)}^2-\langle r_k,e_k^*\rangle-\langle\nabla R_k^*,\nabla e_k^*\rangle.
\end{aligned}
\end{equation}
According to  \cite[Lemma 1]{alikhanov2015new} as well as Lemma~\ref{proMkj}, we can derive
\begin{equation*}
\begin{aligned}
     \langle L_k^{\alpha,*} e ,e_k^*\rangle 
     & = \frac{1}{\Gamma(1-\alpha)}\sum_{j=1}^k [\mathbf M]_{k,j}\langle (e^j-e^{j-1}), (1-\alpha/2)e^k+\alpha/2 e^{k-1} \rangle \\
     & \ge \frac{1}{2\Gamma(1-\alpha)}\sum_{j=1}^k [\mathbf M]_{k,j}\left(\|e^j\|_{L^2(\Omega)}^2-\|e^{j-1}\|_{L^2(\Omega)}^2\right).
\end{aligned}
\end{equation*}
Applying Cauchy-Schwarz inequality in \eqref{eq:errorfun} yields
\begin{equation}\label{eq:4.24}
\begin{aligned}
    & \sum_{j=1}^k [\mathbf M]_{k,j}\left(\|e^j\|_{L^2(\Omega)}^2-\|e^{j-1}\|_{L^2(\Omega)}^2\right)\\
    \le & 2\Gamma(1-\alpha)\|r_k\|_{L^2(\Omega)}\|e_k^*\|_{L^2(\Omega)}+\Gamma(1-\alpha)\| R_k^*\|_{H^1(\Omega)}^2.
    \end{aligned}
\end{equation}

We define a lower triangular $\mathbf P$ matrix such that
\begin{equation*}
    \mathbf P \mathbf M = \mathbf E_{\mathrm L}
\end{equation*}
where 
\begin{equation*}
\mathbf E_{\rm L}=
\begin{pmatrix}
1 & \\
1&1 \\
\vdots& \vdots & \ddots  \\
1& 1&\cdots& 1
\end{pmatrix}.
\end{equation*}
In other words,
\begin{equation*}
    \sum_{l=j}^k[\mathbf P]_{k,l} [\mathbf M]_{l,j}=1,\quad \forall 1\le j\le k\le n.
\end{equation*}
Here $\mathbf P$ is called complementary discrete convolution kernel in the work~\cite{liao2019discrete}. 
It can be easily checked that $[\mathbf P]_{k,l}\geq 0$ due to the monotonicity properties of $\mathbf M$. From \eqref{eq:4.24} we can derive that $\forall 1\leq k\leq n,$
\begin{equation}\label{eq:errormax}
\begin{aligned}
    &\|e^k\|_{L^2(\Omega)}^2\\
    \le & 2 \Gamma(1-\alpha)\sum_{l=1}^k [\mathbf P]_{k,l}\|r_l\|_{L^2(\Omega)}\|e_l^*\|_{L^2(\Omega)}
    +\Gamma(1-\alpha)\sum_{l=1}^k [\mathbf P]_{k,l}\| R_l^*\|_{H^1(\Omega)}^2\\
    \le & 2 \Gamma(1-\alpha) \left(\max_{1\leq l\leq k} \|e_l^*\|_{L^2(\Omega)}\right)\sum_{l=1}^k [\mathbf P]_{k,l}\|r_l\|_{L^2(\Omega)}
    +\Gamma(1-\alpha)\sum_{l=1}^k [\mathbf P]_{k,l}\| R_l^*\|_{H^1(\Omega)}^2,
\end{aligned}    
\end{equation}
where we use 
\begin{equation*}
\begin{aligned}
    &\sum_{l=1}^k [\mathbf P]_{k,l}\sum_{j=1}^l[\mathbf M]_{l,j}\left(\|e^j\|_{L^2(\Omega)}^2-\|e^{j-1}\|_{L^2(\Omega)}^2\right)\\=&\sum_{j=1}^k\left(\|e^j\|_{L^2(\Omega)}^2-\|e^{j-1}\|_{L^2(\Omega)}^2\right)\sum_{l=j}^k [\mathbf P]_{k,l}[\mathbf M]_{l,j}\\=&\sum_{j=1}^k\left(\|e^j\|_{L^2(\Omega)}^2-\|e^{j-1}\|_{L^2(\Omega)}^2\right)=\|e^k\|_{L^2(\Omega)}^2.
\end{aligned}
\end{equation*}
According to Lemma~\ref{lem4.3},
\begin{align*}
     &\Gamma(1-\alpha)\sum_{l=1}^k[\mathbf P]_{k,l}\|r_l\|\\
     &\le C |\Omega|\sum_{l=1}^k[\mathbf P]_{k,l}\bigg([\mathbf M]_{l,1}(t_2^\alpha/\alpha+t_2)+\sum_{j=2}^l([\mathbf M]_{l,j}-[\mathbf M]_{l,j-1})(1+\rho_{j+1})(1+t_{j-1}^{\alpha-3})\tau_j^3\bigg)\\
     &=C |\Omega| \left((t_2^\alpha/\alpha+t_2)+\sum_{j=2}^k(1+\rho_{j+1})(1+ t_{j-1}^{\alpha-3})\tau_j^3\sum_{l=j}^k [\mathbf P]_{k,l}([\mathbf M]_{l,j}-[\mathbf M]_{l,j-1})\right)\\
    &=C|\Omega|\left((t_2^\alpha/\alpha+t_2)+\sum_{j=2}^k (1+\rho_{j+1})(1+ t_{j-1}^{\alpha-3})\tau_j^3[\mathbf P]_{k,j-1}[\mathbf M]_{j-1,j-1}\right)\\
    &=C|\Omega|\left((t_2^\alpha/\alpha+t_2)+ \sum_{j=2}^{k}[\mathbf P]_{k,j-1}[\mathbf M]_{j-1,1}\frac{[\mathbf M]_{j-1,j-1}}{[\mathbf M]_{j-1,1}}(1+\rho_{j+1})(1+ t_{j-1}^{\alpha-3})\tau_{j}^3\right)\\
    & \le C|\Omega|\left((t_2^\alpha/\alpha+t_2)+ \max_{2\le j\le k}\frac{[\mathbf M]_{j-1,j-1}}{[\mathbf M]_{j-1,1}}(1+\rho_{j+1})(1+ t_{j-1}^{\alpha-3})\tau_{j}^3\right)\\
        &\le C|\Omega|\left((t_2^\alpha/\alpha+t_2)+\frac{1}{1-\alpha}\max_{2\le j\le k}(1+\rho_{j+1}) (1+t_{j-1}^{\alpha-3})(t_{j-1}^*)^\alpha\tau_j^3\tau_{j-1}^{-\alpha}\right),
\end{align*}
where $C$ is a constant only depending on $C_m$. 
The last inequality is obtained by the following upper bound of $[\mathbf M]_{j,j}$ and  lower bound of $[\mathbf M]_{j,1}$:
\begin{equation}\label{eq:mjjmj1}
\begin{aligned}
        [\mathbf M]_{j,j}&=c_{j-1}^j+\frac{\sigma^{1-\alpha}}{(1-\alpha)\tau_j^\alpha}\\&=\int_0^1 \frac{\tau_{j-1}^2(2\theta-1)}{\tau_{j}(\tau_{j-1}+\tau_{j})(t_j^*-(t_{j-2}+\theta \tau_{j-1}))^\alpha}\,\dtheta+\frac{\sigma^{1-\alpha}}{(1-\alpha)\tau_j^\alpha}\\
        &\le \frac{1}{\rho_j(1+\rho_j)(\sigma\tau_j)^\alpha}+\frac{\sigma^{1-\alpha}}{(1-\alpha)\tau_j^\alpha}\le \frac{1}{\eta(1+\eta)(\sigma\tau_j)^\alpha}+\frac{\sigma^{1-\alpha}}{(1-\alpha)\tau_j^\alpha},\\
     [\mathbf M]_{j,1}&\ge \frac{\eta}{(1+\eta)\tau_1}\int_0^{t_1}(t_j^*-s)^{-\alpha}\ds\ge\frac{\eta}{(1+\eta)(t_j^*)^\alpha},
\end{aligned}
\end{equation}
where we use (Q1) in Lemma~\ref{proMkj} for the inequality of $[\mathbf M]_{j,1}$.

Using the Taylor formula with integral remainder for $R_j^*$ gives
\begin{equation*}
    R_j^*=-\alpha/2\int_{t_{j-1}}^{t_j^*}(s-t_{j-1})u''(s)\,\ds -(1-\alpha/2)\int_{t_j^*}^{t_j}(t_j-s)u''(s)\,\ds,\quad 1\le j\le k.
\end{equation*}
Under the regularity assumption, we have
\begin{equation*}
   \| R_1^*\|_{H^1(\Omega)}\le C(\tau_1^\alpha/\alpha+\tau_1), \quad \| R_j^*\|_{H^1(\Omega)}\le C(1+ t_{j-1}^{\alpha-2})\tau_j^2, \,\quad2\le j\le k.
\end{equation*}
Then we have 
\begin{align*}
    &\sum_{l=1}^k [\mathbf P]_{k,l}\| R_l^*\|_{H^1(\Omega)}^2\\\le &C\left([\mathbf P]_{k,1}[\mathbf M]_{1,1}\frac{1}{[\mathbf M]_{1,1}}(\tau_1^\alpha/\alpha+\tau_1)^2+\sum_{l=2}^k [\mathbf P]_{k,l}[\mathbf M]_{l,2}\frac{1}{[\mathbf M]_{l,2}}\left((1+ t_{l-1}^{\alpha-2})\tau_l^2\right)^2\right)\\
    \le& C\left(\frac{1}{[\mathbf M]_{1,1}}(\tau_1^\alpha/\alpha+\tau_1)^2+\max_{2\le l\le k}\frac{1}{[\mathbf M]_{l,2}}\left((1+ t_{l-1}^{\alpha-2})\tau_l^2\right)^2\right)\\
    \le& C\left((1-\alpha)\tau_1^\alpha(\tau_1^\alpha/\alpha+\tau_1)^2+\max_{2\le l\le k}(t_l^*)^\alpha((1+ t_{l-1}^{\alpha-2})\tau_l^2)^2\right),
\end{align*}
where we use $[\mathbf M]_{l,2}\ge [\mathbf M]_{l,1}$ and \eqref{eq:mjjmj1}.

Taking the max for $1\le k \le n$ on both sides of \eqref{eq:errormax},  we can derive 
\begin{equation}\label{eq:maxe}
\begin{aligned}
  \max_{1\le k\le n}  \|e_k\|_{L^2(\Omega)}\le& C\bigg((t_2^\alpha/\alpha+t_2)+\frac{1}{1-\alpha}\max_{2\le k\le n}(1+\rho_{k+1}) (1+t_{k-1}^{\alpha-3})(t_{k-1}^*)^\alpha\tau_k^3\tau_{k-1}^{-\alpha}\\
  &+(\tau_1^\alpha/\alpha+\tau_1)\tau_1^{\alpha/2}+\sqrt{\Gamma(1-\alpha)}\max_{2\le k\le n}(t_k^*)^{\alpha/2}(1+ t_{k-1}^{\alpha-2})\tau_k^2\bigg).
\end{aligned}
\end{equation}
The proof is completed.
\end{proof}
In the case of graded mesh with grading parameter $r$,
\begin{equation}\label{eq:tauj}
\begin{aligned}
    t_j  = \left(\frac j K\right)^r T, \quad\tau_j  = t_j -t_{j-1} = \left[\left(\frac j K\right)^r-\left(\frac {j-1} K\right)^r\right]  T,
\end{aligned}
\end{equation}
where $K$ is the total time step number, $1\leq j\leq K,~t_K = T$.
As a consequence, the two terms after $\max$ operations in \eqref{eq:maxe} can be estimated as follows:
\begin{equation}\label{eq:order_graded1}
\begin{aligned}
&(1+\rho_{k+1}) (1+t_{k-1}^{\alpha-3})(t_{k-1}^*)^\alpha\tau_k^3\tau_{k-1}^{-\alpha} 
\leq C  t_{k-1}^{2\alpha-3}\tau_k^{3-\alpha}\\
&= C t_{k-1}^{2\alpha-3}(t_k-t_{k-1})^{3-\alpha}= C(t_{k-1})^\alpha (t_k/t_{k-1}-1)^{3-\alpha}\\
    &=Ct_{k-1}^\alpha ((1+1/(k-1))^r-1)^{3-\alpha}\\
    &\leq C\, r^{3-\alpha}T^\alpha \frac {(k-1)^{r\alpha-(3-\alpha)}}{K^{r\alpha}} 
    =  \frac{C_{T,1}}{K^{\min\{r\alpha,3-\alpha\}}} 
\end{aligned}
\end{equation}
and
\begin{equation}\label{eq:order_graded2}
\begin{aligned}
& (t_k^*)^{\alpha/2}(1+ t_{k-1}^{\alpha-2})\tau_k^2 \leq C
    t_{k-1}^{\alpha-2}\tau_k^2=  C t_{k-1}^{\alpha-2}(t_k-t_{k-1})^2
    = C t_{k-1}^\alpha(t_k/t_{k-1}-1)^2\\
    &=C T^\alpha \left(\frac{k-1}{K}\right)^{r\alpha}((1+1/(k-1))^r-1)^2
    \leq C\, r^{2}T^\alpha \frac {(k-1)^{r\alpha-2}}{K^{r\alpha}} 
    =  \frac{C_{T,2}}{K^{\min\{r\alpha,2\}}}.
\end{aligned}
\end{equation}
In \eqref{eq:order_graded1} and \eqref{eq:order_graded2}, $C_{T,1}$ and $C_{T,2}$ only depend on $T$.
Therefore, if $u$ satisfies the regularity assumptions in Theorem~\ref{thm:err}, then we have the following error estimate of numerical solutions of the L2-1$_\sigma$ scheme on the graded mesh with grading parameter $r$:
\begin{equation}\label{eq:error_graded}
   \max_{1\leq k\leq K} \|u(t_k)-u^k\|_{L^2(\Omega)}\le \frac{\tilde C}{K^{\min\{r\alpha,2\}}}.
\end{equation}
where $\tilde C$ depends on $C_m$ with $m=1,2,3$, $\alpha$ and $\Omega$.
\begin{remark}
When $\alpha\rightarrow 1^{-}$, the constant $\tilde C$ in \eqref{eq:error_graded} will tend to infinity. 
However, using the technique by Chen-Stynes in~\cite{chen2021blow}, one can obtain $\alpha$-robust error estimate in the sense that $\tilde C$ won't tend to infinity when $\alpha \rightarrow 1^{-}$.
\end{remark}
\section{Numerical tests}\label{sect5}
In this section, we provide some numerical tests on the L2-1$_\sigma$ scheme \eqref{eq:sch_sub} of the subdiffusion equation \eqref{eq:subdiffusion}. 

As in~\cite{liao2018second,chen2019error}, the discrete coefficients $a_j^k$ and $c_j^k$ in \eqref{eq:aj} are computed by adaptive Gauss–Kronrod quadrature, to avoid roundoff error problems.

\subsection{1D example}
We first test the convergence rate of an 1D example, where $\Omega=[0,2\pi]$, $T=1$, $u^0(x)\equiv 0$, and $f(t,x)=\left(\Gamma(1+\alpha)+ t^\alpha\right)\sin(x)$.
It can be checked that the exact solution is $u(t,x)=t^\alpha\sin(x)$. 

The graded mesh \eqref{eq:tauj} with grading parameter $r$ and time step number $K$ is adopted in time. 
We use the central finite difference method in space with grid spacing $h=2\pi/10000$. 
The  maximum $L_2$-error is computed by $\max_{1\leq k\leq K} \|u(t_k)-u^k\|_{L^2(\Omega)}$.
Table~\ref{tab1}--\ref{tab3} present the  maximum $L_2$-errors for $\alpha=0.3,\ 0.5,\ 0.7$ and $r = 1,\ 2,\ 2/\alpha,\ 3/\alpha$ respectively. 
It can be observed that the convergence rates are consistent with  \eqref{eq:error_graded} derived from Theorem~\ref{thm:err}.
\begin{table}[!]
\renewcommand\arraystretch{1.1}
\begin{center}
\def\temptablewidth{0.95\textwidth}
\caption{$\max_{1\leq k\leq K} \|u(t_k)-u^k\|_{L^2(\Omega)}$ for the graded meshes with different grading parameters and time step numbers where $\alpha=0.3$.}\label{tab1}
{\rule{\temptablewidth}{1pt}}
\begin{tabular*}{\temptablewidth}{@{\extracolsep{\fill}}ccccccc}
   &$K=40$&$K=80 $ & $K=160 $&$K=320$&$K=480$&$K=640$\\\hline
$r=1$  &2.3600e-2&	2.2505e-2&	2.0661e-2&1.8461e-2&	1.7117e-2&	1.6165e-2\\
order& -- &0.0685   &0.1233   &0.1625   &0.1863   &0.1988 \\\hline 
$r=2$&1.3254e-2&	9.4767e-3&	6.5872e-3&	4.4967e-3&	3.5761e-3&	3.0338e-3\\
order &-- &0.4841   &0.5247   &0.5508   &0.5650&0.5716
\\
\hline 
$r=2/\alpha$&2.7182e-4&	7.4873e-5&	1.9983e-5&	5.2316e-6&	2.3816e-6&	1.3655e-6\\
order & -- &1.8601   &1.9056   &1.9335   &1.9408   &1.9334
\\
\hline 
$r=3/\alpha$&5.6542e-4&	1.5847e-4&	4.2808e-5&	1.1281e-5&	5.1370e-6&	2.9371e-6\\
order &-- &1.8351   &1.8883   &1.9239   &1.9403  &1.9432
\\
 \end{tabular*}
{\rule{\temptablewidth}{1pt}}
\end{center}
\end{table}
\begin{table}[!]
\renewcommand\arraystretch{1.1}
\begin{center}
\def\temptablewidth{0.95\textwidth}
\caption{$\max_{1\leq k\leq K} \|u(t_k)-u^k\|_{L^2(\Omega)}$ for the graded meshes with different grading parameters and time step numbers where $\alpha=0.5$.}\label{tab2}
{\rule{\temptablewidth}{1pt}}
\begin{tabular*}{\temptablewidth}{@{\extracolsep{\fill}}ccccccc}
   &$K=40$&$K=80 $ & $K=160 $&$K=320$&$K=480$&$K=640$\\\hline
$r=1$  &1.8575e-2&	1.4568e-2&	1.1059e-2&	8.2145e-3&6.8534e-3&	6.0116e-3 \\
order& -- &0.3506   &0.3976   &0.4290   &0.4468   &0.4555 \\\hline 
$r=2$&3.9186e-3 &2.0105e-3&	1.0182e-3&	5.1239e-4&	3.4232e-4&	2.5701e-4\\
order &-- &0.9628   &0.9815   &0.9908   &0.9947   &0.9963
\\
\hline 
$r=2/\alpha$&2.2728e-4 &5.8725e-5&	1.4830e-5&	3.7186e-6&	1.6536e-06&	9.3037e-7\\
order & -- &1.9524   &1.9854   &1.9957   &1.9986   &1.9993
\\
\hline 
$r=3/\alpha$&3.5987e-4 & 9.9080e-5	&2.6590e-5&	7.0116e-6&	3.2025e-6&	1.8379e-6\\
order &-- &1.8608   &1.8977   &1.9231   &1.9327   &1.9302
\\
 \end{tabular*}
{\rule{\temptablewidth}{1pt}}
\end{center}
\end{table}
\begin{table}[!]
\renewcommand\arraystretch{1.1}
\begin{center}
\def\temptablewidth{0.95\textwidth}
\caption{$\max_{1\leq k\leq K} \|u(t_k)-u^k\|_{L^2(\Omega)}$ for the graded meshes with different grading parameters and time step numbers where $\alpha=0.7$.}\label{tab3}
{\rule{\temptablewidth}{1pt}}
\begin{tabular*}{\temptablewidth}{@{\extracolsep{\fill}}ccccccc}
   &$K=40$&$K=80 $ & $K=160 $&$K=320$&$K=480$&$K=640$\\\hline
$r=1$  &8.3068e-3&	5.4221e-3&	3.4582e-3&	2.1753e-3&	1.6518e-3&	1.3569e-3 \\
order& -- &0.6154   &0.6488   &0.6688   &0.6790   &0.6836 \\\hline 
$r=2$&7.3797e-4&2.8495e-4&	1.0874e-4&	4.1317e-5&	2.3437e-5&	1.5672e-5\\
order &-- &1.3729   &1.3898   &1.3961   &1.3983   &1.3989
\\
\hline 
$r=2/\alpha$&1.7758e-4&	4.6703e-5&	1.1903e-5&	2.9940e-6&	1.3323e-6&	7.4975e-7\\
order & -- &1.9269   &1.9721   &1.9913   &1.9970   &1.9985
\\
\hline 
$r=3/\alpha$&1.5861e-4&	4.3872e-5&	1.1918e-5&	3.1981e-6&	1.4809e-6&	8.6093e-7\\
order &-- &1.8541   &1.8802   &1.8978   &1.8987   &1.8855
\\
 \end{tabular*}
{\rule{\temptablewidth}{1pt}}
\end{center}
\end{table}

In~\cite{stynes2017error,kopteva2019error}, the authors state that the large value of $r$ in the graded mesh  increases the temporal mesh width near the final time $t = T$ which can lead to large errors. 
Indeed, when $r= 3/\alpha$, the errors seem larger than the case of $r=2/\alpha$, as observed in Table~\ref{tab1}--\ref{tab3}. 
We then propose to use the graded mesh with varying grading parameter $r_j$ (dependent on the time), called $r$-variable graded mesh.
In particular, for this example, we use the following $r$-variable graded mesh
\begin{equation}\label{eq:tauradaptive}
\begin{aligned}
  r_j &=2/\alpha+1.5-\frac{3(j-1)}{K-1},\\
    t_j & = \left(\frac j K\right)^{r_j} T, \quad
    \tau_j = t_j -t_{j-1} = \left[\left(\frac j K\right)^{r_j}-\left(\frac {j-1} K\right)^{r_{j-1}}\right] T.
\end{aligned}
\end{equation}
In Figure~\ref{fig:pterror}, we compare the time steps, the pointwise $L^2$-errors, and the maximum $L^2$-errors of the $r$-variable graded mesh \eqref{eq:tauradaptive} and the standard graded meshes \eqref{eq:tauj} with $r=2/\alpha,~3/\alpha$.
Here we set $\alpha=0.7 $ and for the left and middle subfigures $K=640$.
From the middle of Figure~\ref{fig:pterror}, the maximum $L^2$-error for the $r$-variable graded mesh is smaller than the standard graded meshes with $r =2/\alpha,~3/\alpha$. 
\begin{figure}[!ht]    
\centering
    \includegraphics[trim={3in 0 2in 0},clip,width=1.0\textwidth]{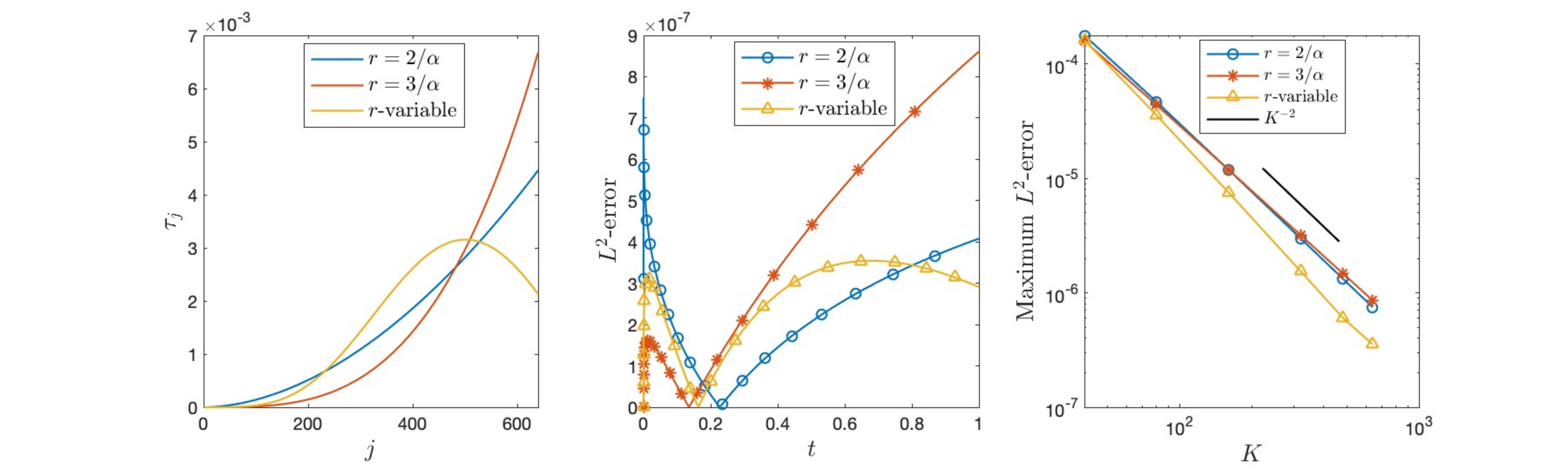}
    \vspace{-0.2in}
    \caption{Time steps (left), pointwise $L^2$-errors (middle), and maximum $L^2$-errors (right) of the L2-1$_\sigma$ scheme in 1D on the $r$-variable graded mesh \eqref{eq:tauradaptive} and the graded meshes \eqref{eq:tauj} with $r=2/\alpha,\ 3/\alpha$ ($\alpha =0.7$). }
    \label{fig:pterror}
\end{figure}

\subsection{2D example}
In the 2D case, we set $f(t,x)=\left(\Gamma(1+\alpha)+2 t^\alpha\right)\sin(x)\sin(y)$ and then the exact solution $u(t,x)=t^\alpha\sin(x)\sin(y)$. In this example, we set periodic boundary condition for the subdiffusion equation.
We take $T=1$ and $\alpha=0.7$. 
Here we use Fourier spectral method in the domain $\Omega=[0,2\pi]^2$ with $256\times 256$ Fourier modes. 
In Figure~\ref{fig:error2d}, we show the pointwise $L^2$-errors (with $K=640$) and the maximum $L^2$-errors of the L2-1$_\sigma$ schemes on the standard graded meshes \eqref{eq:tauj} with $r=2/\alpha$ and the $r$-variable graded mesh \eqref{eq:tauradaptive}. 
One can observe that the $r$-variable graded mesh performs better than the graded mesh for this example.

\begin{figure}
    \centering
    \includegraphics[width=1\textwidth]{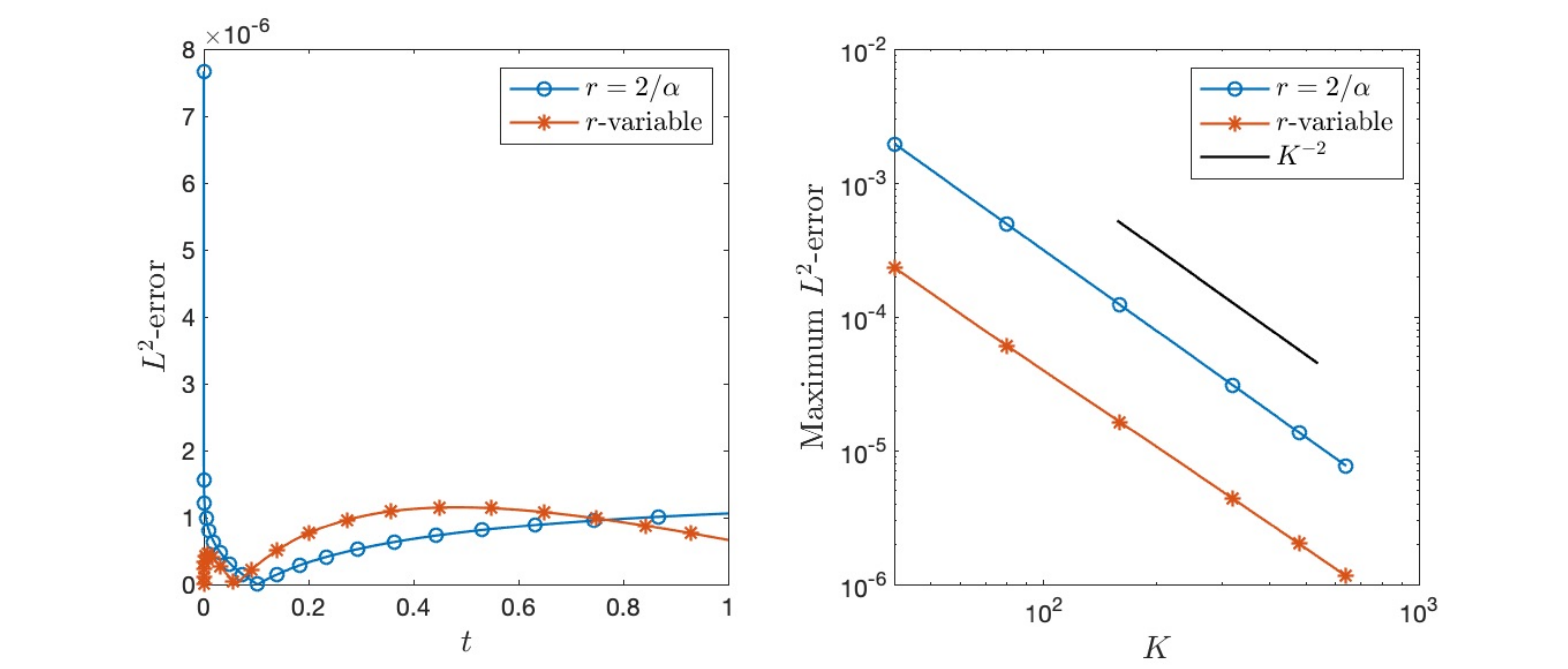}
    \caption{Pointwise $L^2$-errors (left) with $K=640$ and maximum $L^2$-errors (right) of L2-1$_\sigma$ scheme in 2D on the $r$-variable graded mesh \eqref{eq:tauradaptive} and the graded mesh \eqref{eq:tauj} with $r=2/\alpha$ ($\alpha =0.7$). }
    \label{fig:error2d}
\end{figure}

\section*{Declarations}

{\bf Conflicts of interest} The authors declared that they have no conflicts of interest to this work.


\bibliographystyle{spmpsci}
\bibliography{bibfile}

\begin{thebibliography}{10}
\providecommand{\url}[1]{{#1}}
\providecommand{\urlprefix}{URL }
\expandafter\ifx\csname urlstyle\endcsname\relax
  \providecommand{\doi}[1]{DOI~\discretionary{}{}{}#1}\else
  \providecommand{\doi}{DOI~\discretionary{}{}{}\begingroup
  \urlstyle{rm}\Url}\fi

\bibitem{al2019numerical}
Al-Maskari, M., Karaa, S.: Numerical approximation of semilinear subdiffusion
  equations with nonsmooth initial data.
\newblock SIAM Journal on Numerical Analysis \textbf{57}(3), 1524--1544 (2019)

\bibitem{al2022time}
Al-Maskari, M., Karaa, S.: {The time-fractional Cahn--Hilliard equation:
  analysis and approximation}.
\newblock IMA Journal of Numerical Analysis \textbf{42}(2), 1831--1865 (2022)

\bibitem{alikhanov2015new}
Alikhanov, A.A.: A new difference scheme for the time fractional diffusion
  equation.
\newblock Journal of Computational Physics \textbf{280}, 424--438 (2015)

\bibitem{banjai2019efficient}
Banjai, L., L{\'o}pez-Fern{\'a}ndez, M.: Efficient high order algorithms for
  fractional integrals and fractional differential equations.
\newblock Numerische Mathematik \textbf{141}(2), 289--317 (2019)

\bibitem{chen2019error}
Chen, H., Stynes, M.: Error analysis of a second-order method on fitted meshes
  for a time-fractional diffusion problem.
\newblock Journal of Scientific Computing \textbf{79}(1), 624--647 (2019)

\bibitem{chen2021blow}
Chen, H., Stynes, M.: Blow-up of error estimates in time-fractional
  initial-boundary value problems.
\newblock IMA Journal of Numerical Analysis \textbf{41}(2), 974--997 (2021)

\bibitem{gao2014new}
Gao, G.h., Sun, Z.z., Zhang, H.w.: A new fractional numerical differentiation
  formula to approximate the {Caputo} fractional derivative and its
  applications.
\newblock Journal of Computational Physics \textbf{259}, 33--50 (2014)

\bibitem{gorenflo2002time}
Gorenflo, R., Mainardi, F., Moretti, D., Paradisi, P.: Time fractional
  diffusion: a discrete random walk approach.
\newblock Nonlinear Dynamics \textbf{29}(1), 129--143 (2002)

\bibitem{jin2017correction}
Jin, B., Li, B., Zhou, Z.: Correction of high-order {BDF} convolution
  quadrature for fractional evolution equations.
\newblock SIAM Journal on Scientific Computing \textbf{39}(6), A3129--A3152
  (2017)

\bibitem{jin2018numerical}
Jin, B., Li, B., Zhou, Z.: Numerical analysis of nonlinear subdiffusion
  equations.
\newblock SIAM Journal on Numerical Analysis \textbf{56}(1), 1--23 (2018)

\bibitem{jin2020subdiffusion}
Jin, B., Li, B., Zhou, Z.: Subdiffusion with time-dependent coefficients:
  improved regularity and second-order time stepping.
\newblock Numerische Mathematik \textbf{145}(4), 883--913 (2020)

\bibitem{karaa2021positivity}
Karaa, S.: Positivity of discrete time-fractional operators with applications
  to phase-field equations.
\newblock SIAM Journal on Numerical Analysis \textbf{59}(4), 2040--2053 (2021)

\bibitem{kopteva2019error}
Kopteva, N.: Error analysis of the {L}1 method on graded and uniform meshes for
  a fractional-derivative problem in two and three dimensions.
\newblock Mathematics of Computation \textbf{88}(319), 2135--2155 (2019)

\bibitem{kopteva2021error}
Kopteva, N.: Error analysis of an {L}2-type method on graded meshes for a
  fractional-order parabolic problem.
\newblock Mathematics of Computation \textbf{90}(327), 19--40 (2021)

\bibitem{kopteva2020error}
Kopteva, N., Meng, X.: Error analysis for a fractional-derivative parabolic
  problem on quasi-graded meshes using barrier functions.
\newblock SIAM Journal on Numerical Analysis \textbf{58}(2), 1217--1238 (2020)

\bibitem{langlands2005accuracy}
Langlands, T., Henry, B.I.: The accuracy and stability of an implicit solution
  method for the fractional diffusion equation.
\newblock Journal of Computational Physics \textbf{205}(2), 719--736 (2005)

\bibitem{li2022exponential}
Li, B., Ma, S.: Exponential convolution quadrature for nonlinear subdiffusion
  equations with nonsmooth initial data.
\newblock SIAM Journal on Numerical Analysis \textbf{60}(2), 503--528 (2022)

\bibitem{liao2018sharp}
Liao, H.l., Li, D., Zhang, J.: Sharp error estimate of the nonuniform {L1}
  formula for linear reaction-subdiffusion equations.
\newblock SIAM Journal on Numerical Analysis \textbf{56}(2), 1112--1133 (2018)

\bibitem{liao2019discrete}
Liao, H.l., McLean, W., Zhang, J.: A {Discrete} {Grönwall} {Inequality} with
  {Applications to Numerical Schemes for Subdiffusion Problems}.
\newblock SIAM Journal on Numerical Analysis \textbf{57}(1), 218--237 (2019)

\bibitem{liao2018second}
Liao, H.l., McLean, W., Zhang, J.: {A Second-Order Scheme with Nonuniform Time
  Steps for a Linear Reaction-Subdiffusion Problem}.
\newblock Communications in Computational Physics \textbf{30}(2), 567--601
  (2021)

\bibitem{liao2020second}
Liao, H.l., Tang, T., Zhou, T.: A second-order and nonuniform time-stepping
  maximum-principle preserving scheme for time-fractional {Allen-Cahn}
  equations.
\newblock Journal of Computational Physics \textbf{414}, 109,473 (2020)

\bibitem{liao2021energy}
Liao, H.l., Tang, T., Zhou, T.: {An Energy Stable and Maximum Bound Preserving
  Scheme with Variable Time Steps for Time Fractional Allen--Cahn Equation}.
\newblock SIAM Journal on Scientific Computing \textbf{43}(5), A3503--A3526
  (2021)

\bibitem{liao2022discrete}
Liao, H.l., Tang, T., Zhou, T.: Discrete energy analysis of the third-order
  variable-step {BDF} time-stepping for diffusion equations.
\newblock to appear in J. Comput. Math.,  (2022)

\bibitem{liao2021analysis}
Liao, H.l., Zhang, Z.: Analysis of adaptive {BDF}2 scheme for diffusion
  equations.
\newblock Mathematics of Computation \textbf{90}(329), 1207--1226 (2021)

\bibitem{lin2007finite}
Lin, Y., Xu, C.: Finite difference/spectral approximations for the
  time-fractional diffusion equation.
\newblock Journal of Computational Physics \textbf{225}(2), 1533--1552 (2007)

\bibitem{lubich1986discretized}
Lubich, C.: Discretized fractional calculus.
\newblock SIAM Journal on Mathematical Analysis \textbf{17}(3), 704--719 (1986)

\bibitem{lubich1988convolution}
Lubich, C.: {Convolution quadrature and discretized operational calculus. I}.
\newblock Numerische Mathematik \textbf{52}(2), 129--145 (1988)

\bibitem{lubich2004convolution}
Lubich, C.: Convolution quadrature revisited.
\newblock BIT Numerical Mathematics \textbf{44}(3), 503--514 (2004)

\bibitem{lubich1996nonsmooth}
Lubich, C., Sloan, I., Thom{\'e}e, V.: Nonsmooth data error estimates for
  approximations of an evolution equation with a positive-type memory term.
\newblock Mathematics of computation \textbf{65}(213), 1--17 (1996)

\bibitem{lv2016error}
Lv, C., Xu, C.: Error analysis of a high order method for time-fractional
  diffusion equations.
\newblock SIAM Journal on Scientific Computing \textbf{38}(5), A2699--A2724
  (2016)

\bibitem{metzler2000random}
Metzler, R., Klafter, J.: The random walk's guide to anomalous diffusion: a
  fractional dynamics approach.
\newblock Physics reports \textbf{339}(1), 1--77 (2000)

\bibitem{mustapha2014discontinuous}
Mustapha, K., Abdallah, B., Furati, K.M.: A discontinuous {Petrov--Galerkin}
  method for time-fractional diffusion equations.
\newblock SIAM Journal on Numerical Analysis \textbf{52}(5), 2512--2529 (2014)

\bibitem{CSIAM-AM-1-478}
Quan, C., Tang, T., Yang, J.: How to define dissipation-preserving energy for
  time-fractional phase-field equations.
\newblock CSIAM Transactions on Applied Mathematics \textbf{1}(3), 478--490
  (2020).
\newblock \doi{https://doi.org/10.4208/csiam-am.2020-0024}

\bibitem{quan2022h}
Quan, C., Wu, X.: ${H^1}$-stability of an {L}2-type method on general
  nonuniform meshes for subdiffusion equation.
\newblock arXiv preprint arXiv:2205.06060  (2022)

\bibitem{stynes2017error}
Stynes, M., O'Riordan, E., Gracia, J.L.: Error analysis of a finite difference
  method on graded meshes for a time-fractional diffusion equation.
\newblock SIAM Journal on Numerical Analysis \textbf{55}(2), 1057--1079 (2017)

\bibitem{sun2006fully}
Sun, Z.z., Wu, X.: A fully discrete difference scheme for a diffusion-wave
  system.
\newblock Applied Numerical Mathematics \textbf{56}(2), 193--209 (2006)

\bibitem{tang2019energy}
Tang, T., Yu, H., Zhou, T.: On energy dissipation theory and numerical
  stability for time-fractional phase-field equations.
\newblock SIAM Journal on Scientific Computing \textbf{41}(6), A3757--A3778
  (2019)

\bibitem{wang2020high}
Wang, K., Zhou, Z.: High-order time stepping schemes for semilinear
  subdiffusion equations.
\newblock SIAM Journal on Numerical Analysis \textbf{58}(6), 3226--3250 (2020)

\end{thebibliography}


\end{document}